\documentclass{amsart}

\usepackage{color,graphicx,amssymb,latexsym,amsfonts,txfonts}

\usepackage{pdfsync}

\usepackage{hyperref}
\hypersetup{
    colorlinks=true,       % false: boxed links; true: colored links
    linkcolor=blue,          % color of internal links
    citecolor=blue,        % color of links to bibliography
    filecolor=blue,      % color of file links
    urlcolor=blue           % color of external links
}

\begin{document}

\newtheorem{theo}{Theorem}[section]
\newtheorem{prop}[theo]{Proposition}
\newtheorem{coro}[theo]{Corollary}
\newtheorem{lemm}[theo]{Lemma}

\newtheorem{claim}[theo]{Claim}
\newtheorem{conjecture}[theo]{Conjecture}
\theoremstyle{remark}
\newtheorem{rema}[theo]{\bf Remark}
\newtheorem{defi}[theo]{\bf Definition}
\newtheorem{example}[theo]{\bf Example}

\newcommand{\Ch}{\widehat{\mathbb C}}

\title[Sharp upper bounds for the number of fixed points components]{Sharp upper bounds for the number of fixed points components of two and three symmetries of handlebodies}

\author{Ruben A. Hidalgo}
\address{Departamento de Matem\'atica y Estad\'{\i}stica, Universidad de La Frontera, Temuco, Chile}
\email{ruben.hidalgo@ufrontera.cl}

\thanks{Partially  supported by Project FONDECYT 1190001}


\begin{abstract}
An extended Kleinian group whose orientation-preserving half is a Schottky group is called an extended Schottky group. These groups correspond to the real points in the Schottky space. Their geometric structures is well known and it permits to provide information on the locus of fixed points of symmetries of handlebodies.
A group generated by two different extended Schottky groups, both with the same orientation-preserving half, is called a dihedral extended Schottky group.
We provide a structural description of these type of groups and, as a consequence, we obtain sharp upper bounds for the sum of the cardinalities of the connected components of the locus of fixed points of two or three different symmetries of a handlebody.  
\end{abstract}

\maketitle

%%%%%%%%%%%%%%%%%%%%%%%
%%%%%%%%%%%%%%%%%%%%%%%
\section{Introduction}

A uniformization of  a closed Riemann surface $S$ is a triple $(\Delta,\Gamma,P:\Delta \to S)$ where $\Gamma$ is a Kleinian group, $\Delta$ is a $\Gamma$-invariant connected component of its region of discontinuity and $P:\Delta \to S$ is a regular covering map with $\Gamma={\rm Deck}(P)$ (in the classical terminology of Kleinian groups the pair $(\Gamma,\Delta)$ is called a function group; if $\Delta$ is simply connected, it is a B-group). The collection of uniformizations of $S$ is partially ordered where the highest ones are produced when $\Delta$ is simply-connected (for instance, if $S$ has genus at least two, $\Delta$ can be assumed to be the hyperbolic plane and these are the Fuchsian uniformizations). Koebe's retrosection theorem \cite{Bers,Koebe} asserts that $S$ can be uniformized by a Schottky group and, in this case, these provide the Schottky uniformizations. It is known that Schottky uniformizations provide the lowest ones and have not been exploited in this context as for the case of  Fuchsian ones. 

The Schottky space ${\mathcal S}_{g}$, consisting of  the ${\rm PSL}_{2}({\mathbb C})$-conjugacy classes of Schottky groups of rank $g \geq 1$, is a complex manifold. If $g=1$ it is isomorphic to the unit punctured disc and, for $g \geq 2$, it has dimension $3(g-1)$ \cite{Bers,Nag}. On ${\mathcal S}_{g}$ there is a natural real structure, this induced by the complex conjugation map. The fixed points of such a real structure can be identified with the ${\rm PSL}_{2}({\mathbb C})$-conjugacy classes of those extended Kleinian groups (discrete groups of conformal and anticonformal automorphisms of the Riemann sphere $\widehat{\mathbb C}$, necessarily containing anticonformal ones) admitting a Schottky group of rank $g$ as a normal subgroup of finite index.  Examples of these real points are given by the extended Schottky groups (extended Kleinian group whose orientation-preserving half is a Schottky group). The structural description of these groups, in terms of the Klein-Maskit combination theorems  \cite{Maskit:Comb, Maskit:Comb4},  was obtained in \cite{H-G:ExtendedSchottky} (we recall it in Section \ref{Sec:ExtendedSchottky}).

It is well known that a closed Riemann surface $S$ can be described by irreducible complex projective algebraic curves (this is a consequence of Riemann-Roch's theorem in one direction and the implicit function theorem in the other one). Weil's descent theorem \cite{Weil} asserts that such a curve can be chosen to be defined over the reals if and only if $S$ admits a symmetry (an anticonformal involution); in this case $S$ is called symmetric.
As a consequence of the known topological actions of symmetries, if $\tau:S \to S$ is a symmetry, then there is a Schottky uniformization $(\Delta,\Gamma,P:\Delta \to S)$ so that $\tau$ lifts, that is, there is an  anticonformal automorphism $\sigma:\Delta \to \Delta$ so that $\tau P= P\ \sigma$. The automorphism $\sigma$ is the restriction of an extended M\"obius transformation satisfying  $\sigma^{2} \in \Gamma$ and  $\sigma \Gamma \sigma^{-1}=\Gamma$. The group $K=\langle \Gamma,\sigma \rangle$ is an example of an extended Schottky group whose orientation-preserving half is the Schottky group $\Gamma$.
As a consequence, each symmetry on a closed Riemann surface can be realized by an extended Schottky group.

An extended Kleinian group generated by two different extended Schottky groups, both with the same orientation-preserving half, is called a dihedral extended Schottky group. In the above context of symmetries, the dihedral extended Schottky group induces two different symmetries on the surface uniformized by the Schottky group. In \cite{H-M:imaginary} there is provided an example of a Riemann surface admitting two different symmetries both of which cannot be realized by a dihedral extended Schottky group. 

The structural description of those dihedral extended Schottky groups containing no reflections was obtained in \cite{H-M:imaginary}.
In this paper we complete such a description to include the possibility of reflections in the dihedral extended Schottky groups (Theorem \ref{const0}). As an application, we obtain sharp upper bounds on the sum of connected components of fixed points of orientation reversiong involutions on handlebodies of genus at least two (Theorem \ref{sumak=2}). In a forthcomming article we plan to use this structural description to study the connectivity of the real locus of Schottky space.

%%%%%%%%%%%%%%%%%%
%%%%%%%%%%%%%%%%%%%
\section{Symmetries on handlebodies}\label{Sec:handlebodies}
Let $\Gamma$ be a Schottky group of rank $g$, with region of discontinuity $\Omega$, and set $S_{\Gamma}=\Omega/\Gamma$ (a closed Riemann surface of genus $g$).
If ${\mathbb H}^{3}$ denotes the hyperbolic $3$-space, then 
$M_{\Gamma}=({\mathbb H}^{3} \cup \Omega)/\Gamma$ is a handlebody of genus $g$ (we say that $\Gamma$ induces a  Schottky structure on the handlebody).  Its interior $M_{\Gamma}^{0}={\mathbb H}^{3}/\Gamma$ carries a natural complete hyperbolic structure and $S_{\Gamma}$ is its conformal boundary. 

Let $\tau:M_{\Gamma} \to M_{\Gamma}$ be a symmetry, that is,  an order two orientation-reversing homeomorphism  whose restriction to $M_{\Gamma}^{0}$ is a hyperbolic isometry.  The symmetry $\tau$ induces by restriction a symmetry of $S_{\Gamma}$. By lifting $\tau$ to the universal cover, we obtain an extended Schottky group $K_{\tau}$ whose orientation-preserving half is $\Gamma$. As a consequence of the geometrical structure of extended Schottky groups \cite{H-G:ExtendedSchottky}, the locus of fixed points of $\tau$ has at most $g+1$  connected components ($\tau$ is called maximal if it has $g+1$ connected components of fixed points) and each of such connected components is either (i) an isolated point in $M_{\Gamma}^{0}$ or (ii) a $2$-dimensional bordered compact surface (which may or not be orientable) whose border is contained in its conformal boundary $S_{\Gamma}$ (see also \cite{Ka-Mc}). The projection of an isolated fixed point of $\tau$ produces a point in the orbifold $M_{\Gamma}^{0}/\langle \tau \rangle = {\mathbb H}^{3}/K_{\tau}$ admitting a neighborhood which  locally looks like a cone over the projective plane.

If $\widehat{\tau}$ is another symmetry of $M_{\Gamma}$, then (again by the lifting process as above) we obtain  
another extended Schottky group $K_{\widehat{\tau}}$ with $\Gamma$ as its orientation-preserving half. In this way, the group $K$, generated by $K_{\tau}$ and $K_{\widehat{\tau}}$, is a dihedral extended Schottky group. As a consequence of our structural description we obtain sharp upper bounds. In the following ${\mathbb D}_{r}$ denotes the dihedral group of order $2r$.

\begin{theo}\label{sumak=2}
Let $M$ be a handlebody  of genus $g \geq 2$, with a given Schottky structure.
\begin{enumerate}
\item If $\tau_{1}$ and  $\tau_{2}$ are two different symmetries of $M$, $q \geq 2$ is the order of $\tau_{1}\tau_{2}$ and $m_{j}$ is the number of 
connected components of fixed points of $\tau_{j}$, then 
 $$m_{1}+m_{2} \leq 2\left[ \frac{g-1}{q}\right]+4.$$ 
 Moreover, for every integer $q\geq 2$, the above upper bound is sharp for
infinitely many values of $g$.

\item If $\tau_{1}, \tau_{2}$  and $\tau_{3}$ are three different symmetries, $H=\langle \tau_{1},\tau_{2}, \tau_{3}\rangle$ and  $m_{j}$ is
the number of connected components of fixed points of $\tau_{j}$,  then
$$m_{1}+m_{2}+m_{3}\leq \left\{ \begin{array}{ll}
5 & \mbox{if} \; g=2.\\
8 & \mbox{if} \; g=3.\\
g+5 &  \mbox{if} \; g \geq 4 \mbox{ and $H \ncong {\mathbb Z}_{2} \times {\mathbb D}_{r}$ for any $r$.}\\
\dfrac{(r+1)g+5r-1}{r} &  \mbox{if} \; g \geq 4 \mbox{ and $H \cong {\mathbb Z}_{2} \times {\mathbb D}_{r}$ for some $r$.}\\
\end{array}
\right.
$$
Moreover, the above upper bounds are sharp for $g=2,3$ and for infinite many values of $g \geq 4$.
\end{enumerate}
\end{theo}

The above upper bounds are obtained in Section \ref{Sec:pruebasumak=2}. In Section \ref{Sec:4}, we construct explicit examples to see that they are sharp. 
Theorem \ref{sumak=2} asserts the following fact (already observed in \cite{H-M:imaginary} if both symmetries only have isolated fixed points).

\begin{coro}\label{maximales}%\label{maxk=2}
Let $\tau_{1}$ and $\tau_{2}$ be two different symmetries of a handlebody of genus $g \geq 2$, with a Schottky structure, such that $\tau_{j}$ has $m_{j}$ connected components of fixed points. Then $m_{1}+m_{2} \leq g+3$. In particular, (i) a handlebody of genus $g \geq 2$ admits at most one maximal symmetry and (ii) the upper  bound $m_{1}+m_{2}=g+3$ only occurs if $q=2$, that is, when
$\langle \tau_{1}, \tau_{2}\rangle={\mathbb Z}_{2}^{2}$.
\end{coro}

\begin{rema}[Connection to symmetries of Riemann surfaces]
Let $S$ be a closed Riemann surface of genus $g$. 
(1) If $\tau$ is a symmetry of $S$, then each connected component of fixed points of $\tau$ is a simple loop, called an oval and, by Harnack's theorem \cite{Harnack}, the total number of ovals is at most $(g+1)$. 
(2) If $\tau_{1}, \tau_{2}$ are two different symmetries of $S$, then in \cite{BCS} it was observed that the total number of ovals of these two symmetries is bounded above by $2(g-1)/q +4$ (if $q$ is odd) and $4g/q+2$ (if $q$ is even), where $q$ is the order of the product of them. If moreover, $q \geq 3$ and $q$ does not divides $g-1$, then the sharp upper bound is $[ 2(g-1)/q]+3$ (where $[\;]$  stands for the integer part) \cite{Ewa}. We may observe that the upper bounds given in Theorem \ref{sumak=2} are different from the ones obtained for the case of Riemann surfaces. This difference comes from the fact that there Riemann surfaces $S$ admitting two different symmetries which cannot be realized by a common Schottky group uniformizing $S$ \cite{H-M:imaginary}. The number of ovals of two symmetries is at most $2g+2$, in particular, there must be the possibility for $S$ to admit two different maximal symmetries. 
In \cite{Natanzon}, Natanzon proved that if a Riemann surface admits two maximal symmetries, then it is necessarily hyperelliptic. This is again a different situation as that for handlebodies (in a handlebody there is at most one maximal symmetry).
(3) In \cite{N1} there were obtained a sharp upper bound for the number of ovals for three different symmetries of $S$, this being $2(g+2)$. The above bounds are mainly a consequence of the well known structure of non-Euclidian crystallographic (NEC) groups \cite{MacBeath}. This upper bound is again different from that of Theorem \ref{sumak=2}.
\end{rema}

%%%%%%%%%%%%%%%%%
%%%%%%%%%%%%%%%%%
\section{Preliminaries and previous results}\label{Sec:1}
In this section we briefly review several definitions and basic facts we will need in the rest of the paper. More details on these topics may be found, for instance, in \cite{Maskit:book,MT}.

\subsection{Extended Kleinian groups}
We denote by
$\widehat{\mathbb M}$ the group of M\"obius and extended M\"obius transformations (the composition of a M\"o\-bius transformation with the complex conjugation) and by $\mathbb M$ its index
two subgroup of M\"obius transformations. The group $\widehat{\mathbb M}$ can also be viewed, by the Poincar\'e
extension theorem, as the group of hyperbolic isometries of the hyperbolic space ${\mathbb H}^{3}$; in this case,
${\mathbb M}$ is the group of orientation-preserving ones. M\"obius transformations are classified into {\it
parabolic}, {\it loxodromic} (including hyperbolic) and {\it elliptic}. Similarly, extended M\"obius transformations
are classified into {\it pseudo-parabolic} (the square is parabolic), {\it glide-reflections} (the square is
hyperbolic), {\it pseudo-elliptic} (the square is elliptic), {\it reflections} (of order two admitting a circle
of fixed points on $\widehat{\mathbb C}$) and {\it imaginary reflections} (of order two and having no fixed
points on $\widehat{\mathbb C}$). Each imaginary reflection has exactly one fixed point
in ${\mathbb H}^{3}$ and this point determines such reflection uniquely. If $K$ is a subgroup of $\widehat{\mathbb M}$ not contained in ${\mathbb M}$, then $K^{+}=K \cap {\mathbb M}$ is its {\it canonical orientation-preserving subgroup}. 
A  {\it Kleinian group} is a discrete subgroup of ${\mathbb M}$ and  an {\it extended Kleinian group} is a discrete subgroup of $\widehat{\mathbb M}$ necessarily containing extended M\"obius transformations. If $K$ is a (extended) Kleinian group, then its {\it region of discontinuity} is the subset $\Omega$ of $\widehat{\mathbb C}$ composed by the points on which it acts discontinuously. Note that $K$ is an extended
Kleinian groups if and only if $K^+$ is a Kleinian group; both of them with the same region of discontinuity.

\subsection{Klein-Maskit's combination theorems} 

\begin{theo}[Klein-Maskit's combination theorem \cite{Maskit:Comb, Maskit:Comb4}]\label{KMC}
\mbox{}
\\
(1) (Free products) Let $K_{j}$ be a (extended) Kleinian group with region of discontinuity $\Omega_{j}$, for $j=1,2$. Let ${\mathcal F}_{j}$ be a fundamental domain for $K_{j}$ and assume that there is simple closed loop $\Sigma$, contained in the interior of  ${\mathcal F}_{1} \cap {\mathcal F}_{2}$, bounding two discs $\Delta_{1}$ and $\Delta_{2}$, so that, for $j \in \{1,2\}$, $\Sigma \cup \Delta_{j} \subset  \Omega_{3-j}$ is precisely invariant under the identity in $K_{3-j}$. Then (i) $K = \langle K_1, K_2\rangle$ is a (extended) Kleinian group with fundamental domain ${\mathcal F}_{1} \cap {\mathcal F}_{2}$ and $K$ is the free product of $K_{1}$ and $K_{2}$ (ii) every finite order element in $K$ is conjugated in $K$ to a finite order element of either $K_{1}$ or $K_{2}$ and (iii) if both $K_{1}$ and $K_{2}$ are geometrically finite, then $K$ is so.

\smallskip
\noindent
(2) (HNN-extensions) Let $K_{0}$ be a (extended) Kleinian group with region of discontinuity $\Omega$, and let ${\mathcal F}$ be a fundamental domain for $K_{0}$. Assume that there are two pairwise disjoint simple closed loops $\Sigma_{1}$ and $\Sigma_{2}$, both of them contained in the interior of  ${\mathcal F}_{0}$, so that $\Sigma_{j}$ bounds a disc $\Delta_{j}$ such that $(\Sigma_{1} \cup \Delta_{1}) \cap (\Sigma_{2} \cup \Delta_{2})=\emptyset$ and that $\Sigma_{j} \cup \Delta_{j} \subset  \Omega$ is precisely invariant under the identity in $K_{0}$. Let $T$ be either a loxodromic transformation or a glide-reflection so that $T(\Sigma_{1})=\Sigma_{2}$ and $T(\Delta_{1}) \cap \Delta_{2}=\emptyset$. Then (i) $K = \langle K_{0}, f \rangle$ is a (extended) Kleinian group with fundamental domain ${\mathcal F}_{1} \cap (\Delta_{1} \cup \Delta_{2})^{c}$ and $K$ is the HNN-extension of $K_{0}$ by the cyclic group $\langle T \rangle$, (ii) every finite order element of $K$ is conjugated in $K$ to a finite order element of $K_{0}$ and (iii) if $K_{0}$ is geometrically finite, then $K$ is so.
\end{theo}

%%%%%%%%%%%%%%%%%
\subsection{Kleinian $3$-manifolds and their automorphisms}
If $K$ is a Kleinian group and $\Omega$ is its region of discontinuity, then assocated to $K$ is a
$3$-dimensional orientable orbifold $M_{K}=({\mathbb H}^{3} \cup \Omega)/K$; its interior $M^{0}_{K}={\mathbb H}^{3}/K$ has a hyperbolic structure and its conformal boundary $S_{K}=\Omega/K$ has a natural conformal
structure. If $K$ is torsion free, then $M_{K}$ and $M^{0}_{K}$ are orientable $3$-manifolds and $S_{K}$ is a
Riemann surface; we say that $M_{K}$ is a {\it Kleinian $3$-manifold} and that $M_{K}$ and $S_{K}$ are {\it uniformized by $K$}. 
Now, if $K$ is an extended Kleinian group, then the $3$-orbifold $M_{K^{+}}$ admits
the  orientation-reversing homeomorphism $\tau:M_{K^{+}} \to M_{K^{+}}$ of order two induced by $K- K^{+}$ and
$M_{K^{+}}/\langle \tau\rangle=({\mathbb H}^{3} \cup \Omega)/K$. 

Let $\Gamma$ be a torsion free Kleinian group, so $M_{\Gamma}=({\mathbb H}^{3} \cup \Omega)/\Gamma$ is a Kleinian $3$-manifold.
 An {\it automorphism} of $M_{\Gamma}$ is a self-homeomorphism whose restriction to its interior $M_{\Gamma}^{0}$ is a hyperbolic
isometry. An orientation-preserving automorphism is called a {\it conformal automorphim} and an
orientation-reversing one an {\it anticonformal automorphism}. A {\it symmetry} of $M_{\Gamma}$ is an anticonformal
involution. We denote by ${\rm Aut }(M_{\Gamma})$ the group of automorphisms of $M_{\Gamma}$ and by ${\rm Aut}^{+}(M_{\Gamma})$ the subgroup
of conformal automorphisms. Let $\pi^{0}:{\mathbb H}^{3} \to M_{\Gamma}^{0}$ be the  universal covering induced by $\Gamma$. Clearly, $\pi^{0}$
extends to a universal covering $\pi:{\mathbb H}^{3} \cup \Omega \to M_{\Gamma}$ with $\Gamma$ as the group of Deck
transformations. If $H \subset {\rm Aut }(M_{\Gamma})$ is a finite group and we lift it to the universal covering  space
${\mathbb H}^{3}$ under $\pi^{0}$, then we obtain an (extended) Kleinian group $K$ containing $\Gamma$
as a normal subgroup of finite index such that $H=K/\Gamma$. The group $H$ contains orientation-reversing automorphisms  if and only if
$K$ is an extended Kleinian group.

%%%%%%%%%%%%%%%%%
\subsection{Schottky groups}
The {\it Schottky group of rank $0$} is just the trivial group. A {\it Schottky
group of rank $g \geq 1$} is a Kleinian group $\Gamma$ generated by loxodromic
transformations $A_1,\ldots,A_g$, so that there are $2g$ disjoint simple loops,
$C_1,C'_1,\ldots,C_g,C'_g$, with a $2g$-connected outside ${\mathcal D}\subset \widehat{\mathbb C}$, where
$A_i(C_i)=C'_i$, and $A_i({\mathcal D})\cap {\mathcal D}=\emptyset$, for $i=1,\ldots,g$.
The region of discontinuity $\Omega$ of $\Gamma$ is known to be  connected and dense
in $\widehat{\mathbb C}$, $S_{\Gamma}$ is a closed Riemann surface of genus $g$ and $M_{\Gamma}$ is a handlebody of genus $g$. In this case, 
$M_{\Gamma}^{0}$  carries a geometrically finite complete hyperbolic Riemannian metric with injectivity radius bounded away from zero. 
If $g \geq 2$,  then ${\rm Aut }(M_{\Gamma})$ has order at most $24(g-1)$ and ${\rm Aut }^{+}(M_{\Gamma})$ has
order at most $12(g-1) $ \cite{Z1,Z2}. Each conformal (respectively anticonformal) automorphism of $M_{\Gamma}$
induces a conformal (respectively anticonformal) automorphism of the conformal boundary $S_{\Gamma}$ and the
later determines the former due to the Poincare extension theorem.
As a consequence of Koebe's retrosection theorem \cite{Bers, Koebe}, every closed
Riemann surface is isomorphic to $S_{\Gamma}$ for a suitable Schottky group $\Gamma$.

A Schottky group of rank $g$ can be defined as a purely loxodromic Kleinian group of the second kind which is isomorphic to a free of rank $g$ \cite{Maskit:Schottky groups} and it can also be defined as a purely loxodromic geometrically finite Kleinian group which is isomorphic to a free of rank $g$ (essentially a consequence of the fact that a free group cannot be the fundamental group of a closed hyperbolic  $3$-manifold). It follows that every Kleinian structure on a handlebody is provided by a Schottky group, we called it a {\it Schottky structure}. The geometrically finite hyperbolic structures on the interior of a handlebody, with injectivity radius  bounded away from zero,  are provided by Schottky groups.

Let $M$ be a topological handlebody of genus $g$ and let $H$ be a finite group of homeomorphisms of $M$. It is known that there are: (i) a (extended) Kleinian group $K$, containing as a finite index normal subgroup a Schottky group $\Gamma$ of rank $g$, and (ii)  
an orientation-preserving homeomorphism $f:M \to M_{\Gamma}$, with $f H f^{-1} = K/\Gamma$. This is consequence of the fact that a handlebody is a compression body (see also \cite{Z2}).

%%%%%%%%%%%%%%%%%
\subsection{Extended Schottky groups}\label{Sec:ExtendedSchottky}
An {\it extended Schottky group of rank $g$} is an extended Kleinian group whose orientation-preserving half
is a Schottky group of rank $g$.  In the case that it does not contains reflections ({\it Klein-Schottky groups}) a  geometrical structure description was provided in \cite{H-M:imaginary}.  A geometric  structural description of all extended Schottky groups, in terms of the Klein-Maskit combination thorems, is as follows.

\begin{theo}[\cite{H-G:ExtendedSchottky}]\label{maintheo}
\mbox{}
An extended Schottky group is the free product $($in
the Klein-Maskit combination theorem sense$)$ of the following kind of groups:
\begin{itemize}
\item[(i)] cyclic groups generated by reflections,
\item[(ii)] cyclic groups
generated by imaginary reflections,
\item[(iii)] cyclic groups generated by
glide-reflections,
\item[(iv)] cyclic groups generated by loxodromic
transformations, and
\item[(v)] real Schottky groups $($that is groups generated
by a reflection and a Schottky group keeping invariant the corresponding circle of
its fixed points$)$.
\end{itemize}

Conversely,  a subgroup of $\widehat{\mathbb M}$ constructed using $\alpha$ groups of type
$({\rm  i})$, $\beta$ groups of type $({\rm ii})$, $\gamma$ groups of type $({\rm
iii})$, $\delta$ groups of type $({\rm iv} )$  and $\varepsilon$ groups of
type $({\rm v})$, is an extended Schottky group if and only if
$\alpha + \beta + \gamma + \varepsilon>0$.
If, in addition, the $\epsilon$ real Schottky groups above have
the ranks $r_{1},\ldots, r_{\varepsilon} \geq 1$, then it has rank
$g=\alpha + \beta +2(\gamma + \delta) + \varepsilon -1  +r_1 + \ldots + r_\varepsilon.$
\end{theo}

\begin{coro}\label{fijos de simetria}
Let $K$ be an extended Schottky group constructed, as in Theorem \ref{maintheo}, using $\alpha$
groups of type $({\rm  i})$, $\beta $ groups of type $({\rm ii})$, $\gamma$ groups of type $({\rm iii})$, $\delta
$ groups of type $({\rm iv} )$  and $\varepsilon$ groups of type $({\rm v})$. If $\Gamma$ is its
orientation-preserving half, then $K$ induces a  symmetry of $M_{\Gamma}$ whose connected components of fixed points consist of 
$\alpha$  two dimensional closed discs, $\beta $ isolated points,  and $\varepsilon$ two dimensional non-simply connected compact surfaces.  
In particular, If $\tau$ is a symmetry of a Kleinian manifold homeomorphic to a handlebody of genus $g$, and $n_{0}$
is  the number of isolated fixed points of $\tau$, $n_{1}$ is the number of total ovals in the conformal boundary
and $n_{2}$ is  the number of two-dimensional connected components of the set of fixed points of $\tau$, then
$n_{0}+n_{1}, n_{0}+n_{2} \in \{0,1,\ldots, g+1\}.$
\end{coro}

%%%%%%%%%%%%%%
\subsection{A lifting criteria}
We recall a simple criterion for lifting loops which we will need in Section \ref{Sec:dihedral}. This is a direct consequence of the Equivariant Loop Theorem \cite{M-Y}, whose proof is  based on minimal surfaces, that is, surfaces that minimize locally the area. In \cite{Hidalgo:Auto} there is provided a proof whose arguments is proper to (planar) Kleinian groups. A {\it function group} is a pair $(K,\Delta)$, where $K$ is a finitely generated Kleinian group and $\Delta$ is a $K$-invariant connected component of its region of discontinuity.

\begin{theo}\cite{Hidalgo:Auto, M-Y}\label{lifting}
Let $(K,\Delta)$ be a torsion free function group such that $S=\Delta/K$ is a closed Riemann surface of genus $g \geq 2$. 
Let $P:\Delta \to S$ a regular covering with $K$ as its deck group. If
$H$ is a group of  automorphism of $S$, then it lifts to the above regular planar covering if and only if
there is a collection $\mathcal F$ of pairwise disjoint simple loops on $S$ such that:
\begin{itemize}
\item[(i)] $\mathcal F$ defines the regular planar covering $P:\Delta \to S$; and
\item[(ii)] $\mathcal F$ is invariant under the action of $H$.
\end{itemize}
\end{theo}

%%%%%%%%%%%%%%%%%
\subsection{A counting formula}
Let $K$ be an extended Kleinian group containing a Schottky group $\Gamma$ of rank $g$ as a normal
subgroup of finite index (the last trivially holds if  $g \geq 1$). Let  us
denote by $\theta:K \to H=K/\Gamma$,  the canonical projection. If $\tau \in H$ is an involution 
which is the $\theta$-image of an extended M\"obius transformation, then $\widehat{\Gamma} =\theta^{-1}(\langle
\tau \rangle)$ is an extended Schottky group whose orientation-preserving hals is $\Gamma$.  
By Theorem \ref{maintheo}, $\widehat{\Gamma}$ is constructed using  $\alpha$ reflections, $\beta$ imaginary reflections and $\varepsilon$ real Schottky groups. 
These values can be computed from $\Theta$.  Before, we provide some necessary definitions.

A {\it complete set of symmetries of } $K$ is a maximal collection
${\mathcal C}= \{c_{i} \in K: i \in I\}$ of anticonformal involutions (i.e. reflections and
imaginary reflections) which are non-conjugate in $K$( we shall refer to its elements as to canonical symmetries).  As $K$ is geometrically finite (as it is a finite extension of a Schottky group), ${\mathcal C}$ is finite. 

 For each $i \in I$ we set $I(i) \subset I$ defined by
those $j \in I$ so that $\theta(c_i)$ and $\theta(c_{j})$ are conjugate in $H$ (in particular, $i \in I(i)$).
Note that it may happen that for $j \in I(i)$, $c_{j}$ can be imaginary reflection even if $c_i$
is a reflection and  viceversa (this occurs when $\theta(c_i)$ is a symmetry of $M_{\Gamma}$ whose locus of fixed points has isolated
points and also two-dimensional components). We set by $J(i)$ the subset of $ I(i)$
defined by those $j$ for which  $c_j$ is an imaginary reflection.
We also set by $F(i) \subset I(i) \setminus J(i)$ for those $j$ for which $c_{j}$ has finite centralizer in
$K$ and  set $E(i)=I(i) \setminus (J(i)\cup F(i))$. As $\widehat{\Gamma}$ has finite index in
$K$, a reflection $c \in K$ has an infinite centralizer ${\rm C}( K,c)$  in
$K$ if and only if it has an infinite centralizer in $\widehat{\Gamma}$.

\begin{theo}[\cite{H-G:ExtendedSchottky}]\label{teo1}
Let $K$ be an extended Kleinian group containing a Schottky group $\Gamma$ as a finite index normal
subgroup. Let $\theta:K \to H=K/\Gamma$ be the canonical projection and 
${\mathcal C}= \{c_{i}: i \in I\}$ be a complete set of symmetries of  $K$. Then
$\Gamma_i=\theta^{-1}(\langle \theta(c_i) \rangle)=\langle \Gamma, c_i \rangle$ is an extended Schottky group,
constructed using  $\alpha$ reflections, $\beta$ imaginary reflections and $\varepsilon$ real Schottky groups as
in Theorem {\rm \ref{maintheo}}, where
$\alpha=\sum_{j \in F(i)}
[ {\rm C}(H, \theta(c_j)): \theta({\rm C}(K, c_j))]$,
$\beta=\sum_{j \in J(i)} [ {\rm C}(H, \theta(c_j)):
\theta({\rm C}(K, c_j))]$ and
$\varepsilon=\sum_{j \in E(i)}
[ {\rm C}(H, \theta(c_j)): \theta({\rm C}(K, c_j))]$.
\end{theo}

%%%%%%%%%%%%%%%%
%%%%%%%%%%%%%%%%
\section{Structural picture of dihedral extended Schottky groups}\label{Sec:dihedral}
In this section we provide an structural picture of the dihedral extended Schottky groups, in terms of the Klein-Maskit combination theorems, extending the one obtained in \cite{H-M:imaginary}.  We first start with a simple criteria to determine when an extended Kleinian group is a dihedral extended Schottky group and then we describe some basic examples of these type of groups which will be the main actors in the structural description.

%%%%%%%%%%%%%%%%%
\subsection{A simple criteria}
Let $K$ be a dihedral extended Schottky group, say generated by the extended Schottky groups $K_{1}$ and $K_{2}$ with the same orientation-preserving half given by the Schottky group $\Gamma$. If $\Omega$ is the region of discontinuity of $K$ (which is the same for $K_{j}$ and $\Gamma$), then $K_{j}$ induces a symmetry $\tau_{j}$ on the Riemann surface $S=\Omega/\Gamma$. If $p>1$ is the order of $\tau_{1}\tau_{2}$, then 
there exists a surjective homomorphism $\Phi:K \to {\mathbb D}_{p}=\langle \tau_{1},\tau_{2}\rangle$ such that $K_{j}=\Phi^{-1}(\langle \tau_{j}\rangle)$, $\ker(\Phi)=\Gamma$ and $\Phi(K^{+})=\langle \tau_{1}\tau_{2} \rangle$. The converse is provided by the following.

\begin{prop}\label{necesario}
An extended Kleinian group $K$ is a dihedral extended Schottky
group if and only  there is a surjective homomorphism 
$\varphi  :K \to {\rm D}_{p}=\langle a,b: a^{2}=b^{2}=(ab)^{p}=1\rangle $ (for some positive integer $p>1$) whose
kernel is a Schottky group $\Gamma$ and $\varphi  (K^{+})=\langle ab\rangle$.
\end{prop}
\begin{proof}
One direction was already noted above. For the other, 
as $\varphi  (K^{+})=\langle ab\rangle$, then 
$K_{1}=\varphi  ^{-1}(\langle a \rangle)$ and $K_{2}=\varphi  ^{-1}(\langle b \rangle)$ are extended Kleinian groups. As  $K_{j}$  contains $\Gamma$ as an index two subgroup, its is an  
extended Schottky group (where $\Gamma$ is its index two preserving half). Since $K$ is generated by $K_{1}$ and $K_{2}$, it is a dihedral extended Schottky
group.
\end{proof}

\begin{rema}
If $K$ is a dihedral extended Schottky group, then its limit set is totally disconnected (as $K$ contains a Schottky group of finite index) and contains no parabolic transformations. As a consequence of the classification of funtion groups \cite{Maskit:function2,Maskit:function3,Maskit:function4}, every finitely generated Kleinian group, without parabolic nor elliptic elements and with totally disconnected limit set, is a Schottky group.
\end{rema}

%%%%%%%%%%%%%%
\subsection{Examples: Basic dihedral extended Schottky groups}
\subsubsection{}
Let $K$ be constructed (see Figure \ref{Fig1}), by the Klein-Maskit combination theorem, as a free product of $\alpha$ cyclic groups generated by a reflection, $\beta$ cyclic groups generated by an imaginary reflections, $\gamma$ cyclic groups generated by a loxodromic transformation, $\delta$ cyclic groups generated by a glide-reflection, $\rho$ cyclic groups generated by an elliptic transformation of finite order and $\eta$ groups generated by an elliptic transformation of finite order and a loxodromic transformation commuting with the elliptic one and $\kappa$ groups generated by an elliptic transformation of finite order and a glide-reflection transformation both sharing their fixed points (so the glide reflection conjugates the elliptic into its inverse).

Let $\Omega$ be the region of discontinuity of $K$. Then (i) $\Omega/K$ is the connected sum of $\beta+2\gamma+2\delta+2\kappa+2\eta$ real projective planes (if $\beta+\delta+\kappa>0$) or a genus $\gamma+\eta$ orientable surface (if $\beta=\delta=\kappa=0$), 
 with exactly $\alpha$ boundary components and with $2\rho$ cone points in the interior (in pairs of the same cone order); and 
 (ii) $\Omega/K^{+}$ is a closed Riemann surface of genus $\widetilde{g}=\alpha+\beta+2(\gamma+\delta+\kappa+\eta)-1$ with $4\rho$ cone points (in cuadruples with the same cone order). 
As a consequence of Proposition \ref{necesario},  $K$ will be a dihedral extended Schottky group if and only if either (i) $\alpha+\beta+\delta+\kappa \geq 2$ or (ii) $\alpha+\beta+\delta+\kappa=1$ and $\gamma+\eta \geq 1$.  In this case, that proposition, one may take $p$ as the lest common multiple of the orders of the $\rho+\eta+\kappa$ elliptic generators. The $\gamma+\eta$ loxodromic transformations are send to the identity or powers of $ab$, the reflections, imaginary reflections and glide-reflections are sent to either $a$ or $b$ and the elliptic elements are sent to powers of $ab$ (preserving the order).

\begin{figure}[h]
\begin{center}

\includegraphics[width=8cm,keepaspectratio=true]{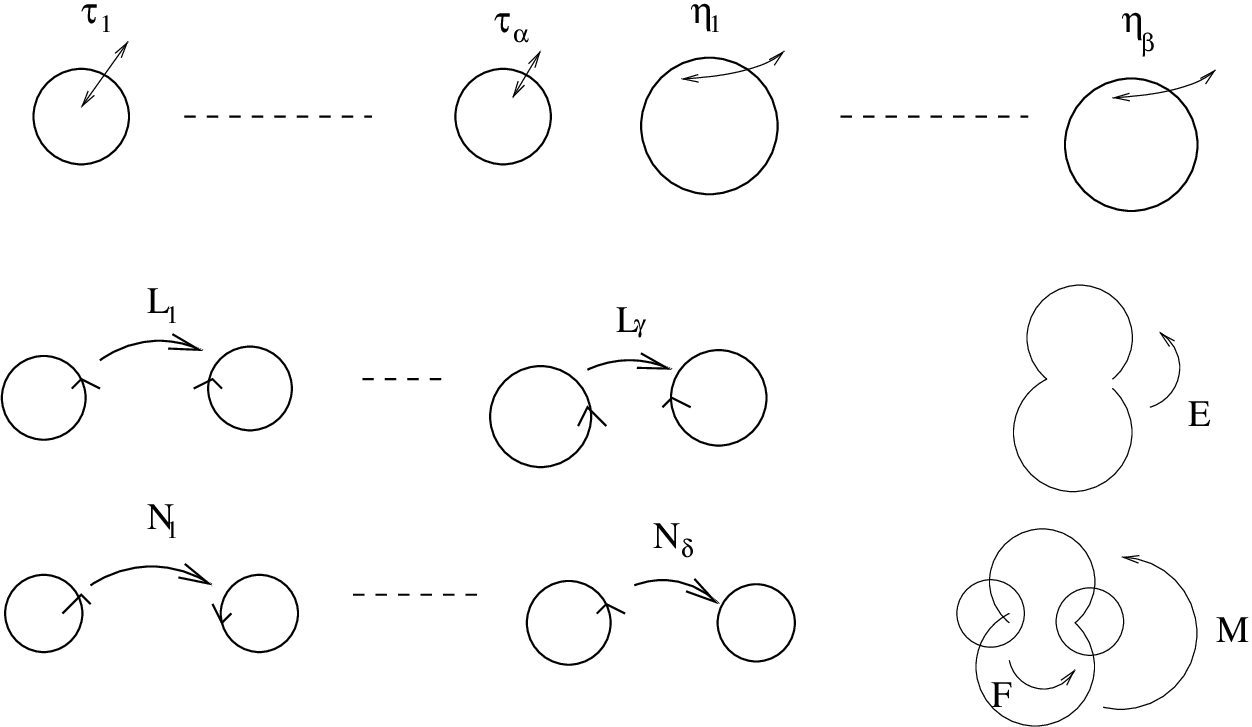}
\caption{Examples of dihedral extended Schottky groups: $\tau_{i}$ reflections, $\eta_{i}$ imaginary reflections, $L_{i}$ loxodromics, $N_{i}$ glide-reflections, $E,F$ elliptics, $M$ either loxodromic or glide-reflection.}
\label{Fig1}
\end{center}
\end{figure}

%%%%%%%%%%%%%
\subsubsection{Basic extended dihedral groups}\label{ejemploBEDG}
Let $C \subset \widehat{\mathbb C}$ be a circle and let $\sigma$ be the reflection on it. Consider a (finite) non-empty collection of parwise disjoint closed discs $\Delta_{j}$, each one with boundary circle $\Gamma_{j}$ being orthogonal to $C$. We denote by $Int(\Delta_{j})$ and $Ext(\Delta_{j})$ the interior and exterior, respectively, of $\Delta_{j}$.

\begin{enumerate}
\item Take $2r$ of these circles (it could be $r=0$), say $\Gamma_{1}, \cdots,\Gamma_{2r}$, and loxodromic transformation $L_{1},\ldots,L_{r}$ such that $L_{j}(\Sigma_{j})=\Sigma_{r+j}$, $L_{j}(Int(\Delta_{j})) \cap Int(\Delta_{j+r})=\emptyset$ and $L_{j}$ commuting with $\sigma$. (The group generated by $\sigma$ and $L_{1},\ldots,L_{r}$ is a real Schottky group).

\item Now, for some each others $\Gamma_{i}$, we consider an elliptic transformation $E_{i}$, with both fixed point on $C$, such that $E_{i}(Ext(\Delta_{i})) \subset Int(\Delta_{i})$ and  $\sigma E_{i} \sigma=E_{i}^{-1}$ (each $E_{i}\sigma$ is a reflection).

\item For others $\Gamma_{k}$, we consider an elliptic transformation $F_{k}$, with both fixed point on $C$, such that $F_{k}(Ext(\Delta_{k})) \subset Int(\Delta_{k})$ and a loxodromic transformation $M_{k}$, such that $M_{k} F_{k}=F_{k}M_{k}$, $M_{k}\sigma=\sigma M_{k}$ and $\sigma F_{k} \sigma=F_{k}^{-1}$ (both fixed points of $F_{k}$ are also the fixed points of $M_{k}$).

\item For others $\Gamma_{s}$ we take an elliptic transformation of order two $D_{s}$ (whose fixed points are not on $C$) keeping invariant $\Gamma_{s}$ (so permuting both discs bounded by it) and commuting with $\sigma$.

\item Finally, for each of  the rest of the circles $\Sigma_{l}$ we consider the reflection $\tau_{l}$ on it (so it commutes with $\sigma$).
\end{enumerate}

The group $K$ generated by $\sigma$ and all of the above transformations is an extended Kleinian group (by the Klein-Maskit combination theorem) called a basic extended dihedral group (see Figure \ref{Fig2}).

If $\Omega$ is the region of discontinuity of $K$, then $S=\Omega/K$ is a bordered (orbifold) Klein surface, which is orientable if and only if the loxodromic transformations $L_{k}$ keep invariant each of the two discs bounded by $C$. Each elliptic element $E_{i}$ produces two cone points on the border, both with cone order the order of $E_{i}$. Each $D_{s}$ produces a cone point of order two in the interior. Each reflection in (5) produces two cone points of order two in the boundary. In particular, $\Omega/K^{+}$ provides of a compact orientable orbifold with some even number of cone points admitting a symmetry permuting them.

As a consequence of Proposition \ref{necesario}, $K$ is a dihedral extended Schottky group.  In this case, we take $p$ as the least common multiple of the orders of all elliptic transformations $E_{i}$, $F_{k}$ and $D_{s}$. We send the loxodromic transformations $L_{j}$ and $M_{k}$ to the identity, the reflection $\sigma$ to $a$ and the reflections $\tau_{l}$ to $b$ (for instance).

\begin{figure}[h]
\begin{center}

\includegraphics[width=8cm,keepaspectratio=true]{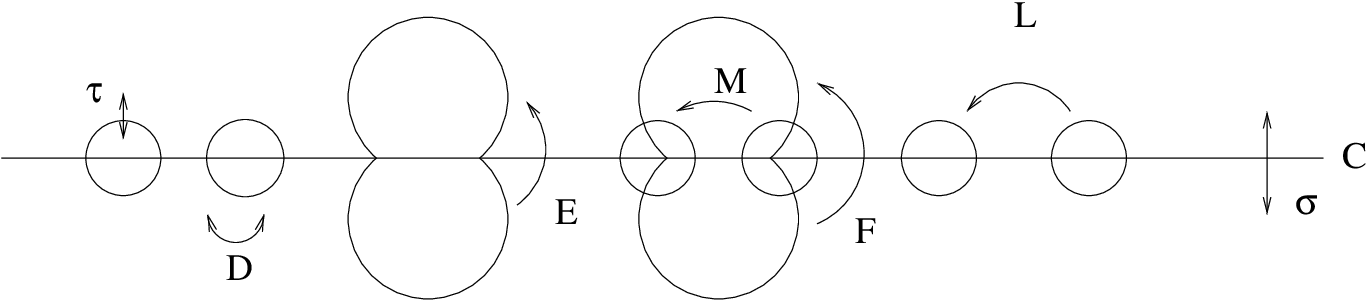}
\caption{Examples of a basic extended dihedral group.}
\label{Fig2}
\end{center}
\end{figure}

%%%%%%%%%%%%%%%%%%%%%%%%
\subsection{The structural description}
\begin{theo}\label{const0}
\mbox{}
$(1)$ A dihedral extended Schottky group is the free product (in the sense of the Klein-Maskit combination theorems) of the following groups (see Figures \ref{Fig1} and \ref{Fig2})
\begin{itemize}
\item[(i)] $\alpha$ cyclic groups generated by reflections;
\item[(ii)] $\beta$ cyclic groups generated by imaginary reflections;
\item[(iii)] $\gamma$ cyclic groups generated by loxodromic transformations;
\item[(iv)] $\delta$ cyclic groups generated by glide-reflections;
\item[(v)] $\rho$ cyclic groups generated by an elliptic transformation of finite order;
\item[(vi)] $\eta$ groups generated by an elliptic transformation of finite order and a loxodromic transformation commuting with the elliptic one;
\item[(vii)] $\kappa$ groups generated by an elliptic transformation of finite order and a glide-reflection transformation both sharing their fixed points (so the glide reflection conjugates the elliptic into its inverse);
\item[(viii)] $\varepsilon$ basic extended dihedral groups $K_{1},\ldots,K_{\varepsilon}$;
\end{itemize}
such that $\alpha+\beta+\delta+\kappa+\varepsilon>0$.

\smallskip \noindent
$(2)$ An extended function group constructed as in $(1)$, is a dihedral extended Schottky
group if and only if either: (i) $\alpha+\beta +\delta +\kappa+\varepsilon \geq 2$ or (ii) $\alpha+\beta +\delta + \kappa+\varepsilon =1$ and $\gamma+\rho+\eta>0.$
\end{theo}

\begin{rema}\label{observa1}
(1) The condition $\alpha+\beta+\delta+\kappa+\varepsilon>0$ in part (1) of Theorem \ref{const0} is needed for the constructed group to have orientatation-reversing elements.

(2) Let $K$ be an extended Kleinian group, constructed as a free product (in the sense of the Klein-Maskit combination theorem) using $\alpha$ groups of type (i), $\beta $
groups of type (ii) $\gamma$ groups of type (iii), $\delta$ groups of type (iv), $\rho$ groups of type (say of orders $l_{1},\ldots,l_{\rho}$) (v), $\eta$ groups of type (vi), $\kappa$ groups of type (vii) 
and $\varepsilon$ basic extended dihedral groups $K_{1},\ldots,K_{\varepsilon}$. 
Assume that the corresponding orbifold $\Omega_{i}/K^{+}_{i}$ has genus $g_{i}$ and $2r_{i}$ cone points of orders $t_{i1},t_{1i},\ldots, t_{ir_{i}}, t_{ir_{i}}\geq 3$ and $2n_{i}$ cone points of order two, where $\Omega_{i}$ is the region of discontinuity of $K_{i}$.
If $|t|$ denotes the order of the transformation $t$, then
\begin{itemize}
\item[(a)] the region of discontinuity $\Omega$ of $K$ is connected and the orbifold $\Omega/K^{+}$ has genus
$\widetilde{g}=\alpha+\beta+2(\gamma+\delta+\kappa+\eta)+\varepsilon+g_{1}+\cdots+g_{\varepsilon} -1$,
has $4\rho$ conical points of orders $l_{1},l_{1},l_{1},l_{1},\ldots,l_{\rho},l_{\rho},l_{\rho},l_{\rho}$, 
$2(r_{1}+\cdots+r_{\varepsilon})$ conical points, of orders $|t_{11}|$, $|t_{11}|$, $|t_{12}|$,
$|t_{12}|$, \ldots, $|t_{\varepsilon r_{\varepsilon}}|$, $|t_{\varepsilon r_{\varepsilon}}|$ and 
$2(n_{1}+\cdots+n_{\varepsilon})$ conical points of order $2$; and

\item[(b)] if $K$ is a dihedral extended Schottky group containing  a Schottky group $\Gamma$ of rank $g$
as a normal subgroup such that $K/\Gamma \cong {\rm D}_{p}$, then (by the Riemann-Hurwitz formula)
$$g=p(\widetilde{g}-1)+1+2p\sum_{i=1}^{\rho}\left(1-\frac{1}{l_{i}}\right)+p\sum_{i=1}^{\varepsilon}\sum_{k=1}^{r_{i}}\left(1-\frac{1}{|t_{ik}|}\right)+
\frac{p}{2}\sum_{i=1}^{\varepsilon}n_{i}.$$
\end{itemize}

\end{rema}

%%%%%%%%%%%%%%%%%%%
%%%%%%%%%%%%%%%%%%%
\subsection{Proof of Theorem \ref{const0}}\label{Sec:teo2}
Part (2) follows from Proposition \ref{necesario}. We proceed to prove part (1).

\subsubsection{}
Let $K$ be a dihedral extended Schottky group, generated by two different extended Schottky groups
$G_{1}$ and $G_{2}$, such that $G_{1}^{+}=G_{2}^{+}=\Gamma$ (a Schottky group of rank $g$) and let $\widehat{G}^{+}$ be its index two
orientation-preserving half.  All the groups $\Gamma$, $G_1$, $G_2$, $K$ and
$K^{+}$ have the same region of discontinuity $\Omega$. Set $S^+=\Omega/\Gamma$, $S_1=\Omega/G_1$, $S_2=\Omega/G_2$, $S_{+}=\Omega/K^{+}$ and
$S=\Omega/K$. 

\begin{prop}  If 
$k \in K^{+}$ is an elliptic transformation with one fixed point in $\Omega$, then both fixed points of
$k$ belong to $\Omega$. 
\end{prop}
\begin{proof} 
As $[K^{+}:\Gamma]<\infty$ and $\Gamma$ contains no  parabolic elements, the same holds for $K^{+}$.
As $K^{+}$ is a geometrically finite function group, the result follows from  \cite{Hidalgo:MEFP}.
\end{proof}

As $K$ is a finite extension of $\Gamma$, and $S^+$ is a closed Riemann surface,
(i) $S$ is a compact not necessarily orientable orbifold, with a finite number of orbifold points and possible
non-empty boundary, and (ii) $S_{+}$ is a closed Riemann surface with some finite number of orbifold points. 

%%%%%%%%%%%%%%%%%%%
\subsubsection{}
Since $\Gamma$ has index 2 in $G_{i}$, for $i=1, 2$, if $\eta \in G_{i} \setminus \Gamma$, then 
$\eta_{i}^{2} \in G$ and either $\eta$ is: (i) a reflection or (ii) an imaginary reflection or (iii) glide-reflection. 
In particular, $G_{i}$ induces
a symmetry $\tau_{i}$ on $S^+$ with $S_{i}=S^+/\langle\tau_{i}\rangle$ and every orientation-reversing element of $K$ either acts without fixed
points on $\Omega$ or they are reflections.

The group $J=\langle \tau_{1}, \tau_{2}\rangle$ is isomorphic to the dihedral group ${\mathbb D}_{p}$, where $p \geq 2$ is the order of $\tau_{1}\tau_{2}$.
It follows that  $S_{+}=S^{+}/\langle \tau_{1}\tau_{2}\rangle$, $S=S^+/J$ and that every orientation-reversing element of $J$ is conjugate in $J$ to either $\tau_1$ or $\tau_2$.

On $S_{+}$ we have an anticonformal involution induced by $J$, preserving the finite set of orbifold points, so that quotient of
$S_{+}$ by it is $S$.

Let us denote by $\Phi:K\to J=\langle \tau_{1}, \tau_{2}\rangle$ the canonical surjective homomorphism with kernel $\Gamma$ which is defined by sending the elements of $G_{i}\setminus \Gamma$ to the symmetry $\tau_{i}$. It follows that $\Phi^{-1}(\langle \tau_{i}\rangle)=G_{i}$.

\begin{rema}
If $p$ is odd, then $\tau_{1}$ and $\tau_{2}$ are conjugated in $J$ and, in particular, $G_{1}$ and $G_{2}$
are conjugated in $K$.
\end{rema}

\begin{prop}\label{reversing}
Every orientation-reversing element of $K$ is either glide-reflection or an imaginary
reflection or a reflection.
\end{prop}

\begin{proof}
An orientation-reversing element $a \in K$ is a lift of a conjugate of either $\tau_1$ or $\tau_2$,
each one a symmetry of $S^{+}$, so $a^2\in \Gamma$. In this way, either (i) $a^2=1$, in which case $a$ is either a reflection or an
imaginary reflection, or (ii) $a^{2}$ is a loxodromic transformation, in which case it is a glide-reflection. 
\end{proof}

%%%%%%%%%%%%%%%%%%%%%%
\subsubsection{Structural loops and regions}
Let us consider the regular planar Schottky covering $P:\Omega \to S^{+}$. 
As the group $J$ lifts under $P$ (such a lifting is the group $K$), Theorem
\ref{lifting} asserts the existence of a $J$-invariant collection of pairwise disjoint loops $\mathcal F$ in $S^{+}$ which
divides $S^{+}$ into genus zero surfaces (since we are dealing with a Schottky covering) and each loop lifting to a
simple loop on $\Omega$. If $A$ is a connected component of $S \setminus {\mathcal F}$ and $J_{A}$ is its $J$-stabilizer, 
then (as $J$ is a dihedral group) the subgroup $J_{A}$ is either:
\begin{itemize}
\item[(i)] trivial; or
\item[(ii)] a cyclic conformal group generated by a power of $\tau_{1}\tau_{2}$; or
\item[(iii)] a cyclic group of order two generated by a symmetry (which is conjugated to either $\tau_{1}$ or $\tau_{2}$); or
\item[(iv)] a dihedral subgroup of $J$.
\end{itemize}

In either case (iii) or (iv) we have that $A$ is invariant under symmetry $\tau \in J$. If
$\tau$ is either (a) a reflection containing a loop of fixed points in $A$ or (b) an imaginary reflection, then
we may find a simple loop $\beta \subset A$ which is invariant under $\tau$; moreover, if $\tau$ is reflection,
then $\beta$ is formed of only fixed points of it. We may add such a loop and its $J$-translated to ${\mathcal
F}$ without destroying the conditions of Theorem \ref{lifting}. In this way, we may also  assume that ${\mathcal F}$ satisfies the following extra property:

\begin{enumerate}
\item[(v)] Every symmetry in $J_{A}$ is necessarily a
reflection whose circle of fixed points is not completely contained in $A$ (it intersects some boundary loops).
\end{enumerate}

Let us observe that  we may assume ${\mathcal F}$ to be minimal in the sense that there is no proper subcollection of it satisfying (i)-(v) above.
The loops in ${\mathcal F}$ are called the {\it base structure loops} and the connected components of $S^{+} \setminus {\mathcal F}$ are called the {\it base structure
regions}. Both collections are $J$-invariants.

Let $\mathcal G$ be the collection of loops on $\Omega$ obtained by the lifting of those in $\mathcal F$; these are called 
 the {\it structure loops}.  The connected components
of $\Omega \setminus {\mathcal G}$ will be called the {\it structure regions}.  Both collections, ${\mathcal G}$ and the set of structure regions, are $K$-invariant.

%%%%%%%%%%%%%%%%
\subsubsection{Stabilizers of structure regions}
Let $R$ be a structure region and $K_{R}$ be its $K$-stabilizer. As $P:R \to P(R)=A$ is a homeomorphism, $K_{R}$
is isomorphic to $J_{A}$. As a consequence we obtain the following fact.

\begin{prop}\label{regstab}
If $R$ is a structure region, then $K_{R}$ is either:
\begin{itemize}
\item[(i)] trivial;
\item[(ii)] a finite cyclic group generated by an elliptic transformation;
\item[(iii)] a cyclic group of order two generated by a reflection
whose circle of fixed points is not completely contained on $R$;
\item[(iv)] a dihedral group generated two reflections, each one 
with its circle of fixed points intersecting some boundary loop of $R$.
\end{itemize}
\end{prop}

\begin{prop} If $R$ is a structure region with $K_{R}=\{I\}$,
then the restriction to $R$ of the projection map from $\Omega$ to $S=S^{+}/J=\Omega/K$ is a homeomorphism
onto its image.
\end{prop}
\begin{proof}If $k\in K$, then either $k\in K_{R}=\{I\}$,
in which case, $k(R)=R$, or $k\not\in K_{R}$, in which case, $k(R)\cap R =\emptyset$.
\end{proof}

%%%%%%%%%%%%%%%%%%%
\subsubsection{Stabilizers of structure loops}
If $\beta \in {\mathcal G}$, then we denote its $K$-stabilizer as $K_{\beta}$.

As $J^{+}=\langle \tau_{1}\tau_{2}\rangle$ (a cyclic group), the orientation-preserving half of $K_{\beta}$ is either trivial or a finite cyclic
group generated by some elliptic element.

By Proposition \ref{reversing}, an orientation-reversing transformation inside $K$ is
either an imaginary reflection or a reflection or a glide-reflection transformation. As the structure loop $\beta$
is contained in $\Omega$, a glide-reflection cannot belong to $K_{\beta}$.
In particular, the only orientation-reversing transformations in $K$ that keep invariant
some structure loop can be either an imaginary reflection or a reflection. 
Also, $K_{\beta}$ cannot contain two different imaginary reflections, for the product of two distinct imaginary reflections is always
hyperbolic.

The structure loop $\beta$ can be stabilized by a
reflection in two different manners. One is that it fixes it point-wise, that is, $\beta$ is
the circle of fixed points of the reflection. The second one is that $\beta$ is not point-wise
fixed by the reflection, in which case, there are exactly two fixed points of the reflection on it. 
These two points divide the loop into two arcs which are permuted by the reflection. 

Summarizing all the the above is the following.

\begin{prop} \label{loopsestab}
If $\beta \in {\mathcal G}$, then $K_{\beta}$ is either:
\begin{itemize}
\item[(i)] trivial;

\item[(ii)] a cyclic group generated by an elliptic element of finite order whose fixed points are separated by $\beta$;

\item[(iii)] a cyclic group generated by an elliptic element of order two with its fixed points on $\beta$;

\item[(iv)] a cyclic group generated by an imaginary reflection;

\item[(v)] a cyclic group generated by an reflection with $\beta$ as its circle of fixed points;

\item[(vi)] a cyclic group generated by a reflection and $\beta$ containing exactly two fixed points of it;

\item[(vii)] a group generated by a reflection, with exactly two fixed points on $\beta$, and an elliptic involution with these two points as its fixed points.
In this case, the composition of these two is a reflection with $\beta$ as circle of fixed points and  $K_{\beta}$ is isomorphic to ${\mathbb Z}_{2}^{2}$;

\item[(viii)] a group generated by an elliptic involution, whose fixed points are separated by $\beta$, and a reflection with $\beta$ as
its circle of fixed points. In this case,  the composition of these two is an imaginary reflection and $K_{\beta}$ is isomorphic to ${\mathbb Z}_{2}^{2}$;

\item[(ix)] a group generated by an elliptic involution, with both fixed points on $\beta$, and a reflection whose circle of fixed points intersects $\beta$ at two points and separating the fixed points of the elliptic involution. 
In this case, $K_{\beta}$ is isomorphic to ${\mathbb Z}_{2}^{2}$; and

\item[(x)] a group generated by a reflection, with exactly two fixed points on
$\beta$, and an elliptic involution with both fixed points on the circle of fixed points of the reflection.
In this case, $K_{\beta}$ is isomorphic to ${\mathbb Z}_{2}^{2}$;
 \end{itemize}
\end{prop}

%%%%%%%%%%%%%%%%%%%%%%
\subsubsection{}
Later, we will construct a connected compact domain by gluing a finite set of structure regions an loops such that no two of these structure regions are $K$-equivalent. 
If two structure regions share a boundary structural loop, then they are $K$-equivalent if there is some element of $K$ sending a boundary loop of one to a boundary loop of the other. This permits to obtain the following simple fact.

\begin{prop}\label{equivalente}
Let $R$ and $R'$ be any two different structure regions with a common boundary loop $\beta$. Then, they are
$K$-equivalent if and only if either: (i) $K_{\beta}$ contains an element (necessarily of order two) which does not belong
to $K_{R}$ or (ii) there is another boundary loop $\beta'$ of $R$ and an element $k \in
K \setminus K_{R}$ so that $k(\beta')=\beta$ (necessarily a loxodromic or glide-reflection).
\end{prop}

%%%%%%%%%%%%%%%%%%%%%
\subsubsection{}
Next result is  related to those structure regions with $K$-stabilizer being either trivial or a cyclic group generated by
a reflection.

\begin{prop}\label{trivial}
Let $R$ be a structure region with $K_{R}$ being either trivial or a cyclic group
of order two generated by a reflection. If
$\beta$ is a boundary loop of $R$ such that $K_{R} \cap K_{\beta}=\{I\}$, then there is
a non-trivial element $k \in K \setminus K_{R}$ so that $k(\beta)$ still a boundary loop of $R$.
\end{prop}
\begin{proof}
The hypothesis on $K_{R}$ asserts that the only possibilities in 
Proposition \ref{loopsestab} for $K_{\beta}$ are (i), (iii), (iv) and (v). 
In all of these cases, with the exception of (i), $K_{\beta}$ contains an element outside $K_{R}$.
In case (i) $K_{\beta}$ is trivial. The projection of $\beta$ on $S^{+}$ is a simple loop
$\beta_{*}$ which has trivial $J$-stabilizer. We have that $\beta$ is free homotopic to the product of the
other boundary loops of $R$. If none of the other boundary loops of $R$ is equivalent to $\beta$ under
$K$, then we may delete $\beta_{*}$ and its $J$-translates from ${\mathcal F}$,  contradicting the
minimality of ${\mathcal F}$.
\end{proof}

%%%%%%%%%%%%%%%%%%%%%%%
\subsubsection{Structure loops with non-trivial conformal stabilizer}
Let $R$ be a structure region with $K_{R}$ neither trivial or a cyclic group generated by a reflection. 
Proposition \ref{regstab} asserts that $K_{R}$ is either a cyclic group generated by an elliptic transformation or it is a dihedral group generated by two reflections.
In either situation, $H=K_{R}^{+}$ is a non-trivial elliptic cyclic group. 
The structure region  $R$ has either $0$, $1$ or $2$ boundary structure loops being stabilized by $H$. The other boundary structure loops, if any, have trivial $H$-stabilizers.  

If no structure loop on the boundary of $R$ is stabilized by
$H$, then both fixed points of $H$ lie in $R$. Also, a boundary structure loop is stabilized by $H$ if and
only if it separates the fixed points of $H$. 

If $K_{R}$ is a dihedral group and $R$ contains only one of the fixed points of $H$, then 
such a fixed point is fixed by the reflections in $K_{R}$. This, in particular, asserts that each reflection in $K_{R}$ will commute with the elements of $H$; this obligates to have $K_{R} \cong {\mathbb Z}_{2}^{2}$.

\begin{prop} Let $R$ be a structure region with $K_{R}^{+}=H$ being non-trivial.  If there is a fixed point of $H$
in $R$, then both fixed points of $H$ lie in $R$.
\end{prop}
\begin{proof}
Suppose there is only one fixed point in $R$ of the cyclic group $H$.
Then there is a unique structure loop $\beta$ on the boundary of $R$ stabilized by $H$. Every other structure
loop, on the boundary of $R$, has $H$-stabilizer the identity.
If $K_{R}=H$, then it follows that if were to fill in the discs bounded by the other structure loops on
the boundary of $R$, then $\beta$ would be contractible; that is, if we delete the projection of $\beta$ and their
$J$-translates from our list of base structure loops, this would leave unchanged the smallest normal subgroup
containing the base structure loops raised to appropriate powers. Since we have chosen our base structure loops
to be minimal, this cannot be.
If $K_{R} \neq H$, then we have a reflection $\tau \in K_{R}$ whose circle of fixed points is
not completely contained in $R$, and $K_{R}=\langle H,\tau\rangle$ is a dihedral group. In this case, 
the fixed point of $H$ contained in $R$ is also fixed by $\tau$ (so both fixed points of $H$ are fixed
by $\tau$ as $\tau H \tau=H$), $K_{R} \cong {\mathbb Z}_{2}^{2}$ and  $H \cong {\mathbb Z}_{2}$. In this case, we
may also delete the projection of $\beta$ and its $J$-translates from ${\mathcal F}$ in order to get a contradiction
to the minimality of ${\mathcal F}$.
\end{proof}

The previous result asserts that a non-trivial elliptic transformation in $K_{R}^{+}$ either has both fixed points on the structure region $R$ or none of them belong to it. In the last case, there are (exactly) two boundary structure loops of $R$, each one invariant under such an elliptic transformation.

\begin{prop}
Let $R$ be a structure region with non-trivial $K_{R}^{+}=H$.
If $\beta_{1}, \beta_{2}$ are two different boundary loop of $R$ which are invariant
under $H$, then there is a (non-trivial) element $k \in K$ so that $k(\beta_{1})=\beta_{2}$ (such an element is either loxodromic or glide-reflection).
\end{prop}
\begin{proof}
Let us assume that there is no such element of $K$ as desired and let $R_{*}$ be the other structure region sharing $\beta_{2}$ in its boundary. On the region $R_{*}$ there is nother boundary loop $\beta_{3}$ which is invariant under $H$. All other boundary loops of $R \cup R_{*}$ (with the exception of $\beta_{1}$, $\beta_{2}$ and $\beta_{3}$) have trivial $K$-stabilizers. In particular, they are not $K$-equivalents to $\beta_{1}$, $\beta_{2}$ and $\beta_{3}$. Also, $\beta_{2}$ is neither $K$-equivalent to $\beta_{1}$ and $\beta_{3}$ (by our assumption). If we project the region $R \cup R_{*} \cup \beta_{2}$ on $S^{+}$, we obtain an homeomorphic copy and the projected loop from $\beta_{2}$ is not $J$-equivalent to none of its boundary loops. In particular, we may delete it (and its $J$-translates) obtaining a contradiction to the minimality of ${\mathcal F}$.
\end{proof}

%%%%%%%%%%%%%%%%%%%%%%%%%%%%%%%%
\subsubsection{Structure regions with trivial stabilizers}
Let $R$ be a structure region with trivial stabilizer
$K_{R}=\{I\}$. As consequence of Proposition \ref{trivial}, every other structure region is necessarily $K$-equivalent
to $R$. It follows that $R$ is a fundamental domain for $K$ and the boundary loops are paired by
either reflections, imaginary reflections, loxodromic transformations or glide-reflections. In this case we obtain that $K$ is the free product, by the Klein-Maskit combinatioin theorem, of groups of types (i)-(iv) as described in the theorem.

%%%%%%%%%%%%%%%%%%%%%%%
\subsubsection{Structure regions with non-trivial stabilizers}
Let us now assume there is no structural region with trivial $K$-stabilizer.

Proposition \ref{equivalente}, and the fact that $S$ is compact and connected, permits us to construct a
connected set $\widehat{R}$ obtained as the union of a finite collection of $K$-non-equivalent structure regions
(each of them has non-trivial $K$-stabilizer) together their boundary structure loops. 

%Some of the boundary loops of $\widehat{R}$ have a non-trivial $K$-stabilizer which sends $\widehat{R}$ at the other side of the
%boundary loop an the other boundary loops are equivalent in pairs under $K$.

\smallskip
\noindent (1) Let $R \subset \widehat{R}$  be a structure region with $K_{R}=\langle \tau \rangle$, where $\tau$ is a reflection.
The circle of fixed points $C_{\tau}$ of $\tau$
intersects some structure loops in the boundary of $R$. If $\gamma=C_{\tau} \cap R$, then $R-\gamma$ consists of two domains $R_{1}^{0}$ and $R_{2}^{0}$. Set 
 $R^{*}_{j}=R^{0}_{j} \cup \gamma \subset R$. 
 
By Proposition \ref{trivial}, the structural loops contained in the interior of $R^{*}_{1} \cup R^{*}_{2}$ are paired by either
reflections, imaginary reflections,  loxodromic transformations or glide-reflection transformations. This process provides free products (in the sense of the Klein-Maskit combination theorem) of groups of types (i)-(iv) as described in the theorem.

The other boundary loops of $R$ intersect $C_{\tau}$. Let $\beta$ be any structure loop in the boundary
of $R$ which intersects $C_{\tau}$, necessarily at two points. 

(1.1) If $\beta$ belongs to the boundary of
$\widehat{R}$, then we already noted that either
\begin{itemize}
\item[(a)] there is an involution $k \in K$ (conformal or anticonformal) with $k(\beta)=\beta$ and
$k(\widehat{R}) \cap \widehat{R}=\beta$; or
\item[(b)] there is another boundary loop
$\beta'$ of $\widehat{R}$ and an element $\sigma \in K$ (which is either loxodromic or a glide-reflection) so that $\sigma(\beta)=\beta'$ and
$\sigma(\widehat{R}) \cap \widehat{R}=\beta'$. If $\sigma$ is a glide-reflection, then $\tau\sigma$ is a loxodromic with the same property.
\end{itemize}

In any of the two situations above, we perform the HNN-extension (in the sense of Kleini-Maksti combination theorem) to produce factors of type (2), (4) or (5) in the descriotion of basic extended dihedral groups as in Section \ref{ejemploBEDG}.

(1.2) If $\beta$ is in the interior of $\widehat{R}$, then we have another structure region $R' \subset \widehat{R}$
with $\beta$ as one of its border loop. In this case, as consequence of Proposition \ref{equivalente}, we should
have $K_{\beta}=\langle \tau\rangle$ and $\tau \in K_{R'}$.

The above process produces basic extended dihedral groups using the circle $C_{\tau}$ and its reflection $\tau$.

\smallskip
\noindent (2) Let $R \subset \widehat{R}$  be a structure region with $K_{R}=\langle \phi \rangle$, where $\phi$ is an elliptic transformation. In this case we have two possibilities: either (i) both fixed
points of $\phi$ belong to $R$ or (ii) there are two boundary loops $\beta_{1}$ and $\beta_{2}$ of $R$, each one
invariant under $\phi$, and there is some $k \in K$ with $k(\beta_{1})=\beta_{2}$,
$k(\widehat{R}) \cap \widehat{R}=\beta_{2}$. We have that each other boundary loop $\beta$ of $R$ is either
(i) the boundary of another structure region inside $\widehat{R}$ (so it has trivial stabilizer in $K$)
or (ii) it belongs to the boundary of $\widehat{R}$ (in which case we have either (a) or (b) above). This process produces the free product of groups of types (v), (vi) and (vii) as in the theorem.

\smallskip
\noindent (3) Assume one of the structure regions $R \subset \widehat{R}$ has stabilizer a dihedral group. In this case, $K^{+}_{R}$ is a non-trivial cyclic group generated by an elliptic transformation. We may proceed similarly as above and we again produce  basic extended dihedral groups.

%%%%%%%%%%%%%%%%%
\subsubsection{}
All the above together permits to see that $K$ can be constructed, by the Klein-Maskit combination theorem, by using the groups as stated in the theorem.
$\blacksquare$

%%%%%%%%%%%%%%%%
%%%%%%%%%%%%%%%%
\section{Proof of Theorem \ref{sumak=2}}\label{Sec:pruebasumak=2}
Let $\Gamma$ be a Schottky group of rank $g \geq 2$ with $M=({\mathbb H}^{3} \cup \Omega)/\Gamma$, where $\Omega$
is the region of discontinuity of $\Gamma$.

\subsection{Proof of part (1)}
By lifting both $\tau_{1}$ and $\tau_{2}$ to ${\mathbb H}^{3}$
we obtain an extended Kleinian group $\widehat{K}$ containing $\Gamma$ as a normal subgroup so that
$\widehat{G}=\widehat{K}/\Gamma =\langle \tau_{1}, \tau_{2} \rangle \cong {\mathbb D}_{q}$, that is, $\widehat{K}$ is a dihedral extended Schottky group. 
Let $\theta:\widehat{K} \to \widehat{G}$ be the canonical surjection and  by $c_{1},\ldots, c_{r} \in
\widehat{K}$ a complete set of symmetries.
By Theorem \ref{const0}, $\widehat{K}$ is constructed using reflections $\zeta_{1}$,\ldots, $\zeta_{\alpha}$, imaginary
reflections $\eta_{1}$,\ldots, $\eta_{\beta}$, $\gamma$ cyclic groups generated by loxodromic transformations,
$\delta$ cyclic groups generated by glide-reflections, $\rho$ cyclic groups generated by an elliptic transformation of finite order,
$\eta$ groups generated by an elliptic transformation of finite order and a loxodromic transformation commuting with it, 
$\kappa$ groups generated by an elliptic transformation of finite order and a glide-reflection (conjugating the elliptic into its inverse)
and $\varepsilon$ basic extended dihedral groups $K_{1}$,\ldots,
$K_{\varepsilon}$. 

Each $K_{i}$ is generated by a reflection $\sigma_{i}$ and some other loxodromic and elliptic
transformations, where some of them do not commute with loxodromic ones, say $t_{i1}$,\ldots, $t_{im_{i}}$ (all of them of order at least $3$ and not commuting with $\sigma_{i}$)
and some $f_{i}$ imaginary reflections and reflections (each of them commuting with $\sigma_{i}$); let us denote
these involutions by $\sigma_{i1}$,\ldots, $\sigma_{if_{i}}$. 
By the Riemann-Hurwitz formula (see Remark \ref{observa1}), 
$g \geq q(\widetilde{g}-1)+1+\frac{q}{2}\sum_{i=1}^{\varepsilon}f_{i}$,
where
$\widetilde{g}=\alpha +\beta +2(\gamma+\delta+\kappa+\eta)+ \varepsilon +g_{1}+\cdots+g_{\varepsilon}-1 \geq \alpha + \beta + \varepsilon -1.
$
So, it follows that
$$
\frac{g-1}{q}+2 \geq \widetilde{g}+1+\frac{1}{2}(f_{1}+\cdots+f_{\varepsilon})\geq \alpha+\beta+\varepsilon+\frac{1}{2}(f_{1}+\cdots+f_{\varepsilon}).$$

Also, note that a complete set of symmetries for $\widehat{K}$ is given by $\zeta_{1}$,\ldots, $\zeta_{\alpha}$,
$\eta_{1}$,\ldots, $\eta_{\beta}$, $\sigma_{i}$, $\sigma_{ik}$, where  $k=1,\ldots,f_{i}$  and  $i=1,\ldots,\varepsilon$. 
If $c$ denotes any of the above symmetries, then $\langle c \rangle < {\rm C}
(\widehat{K},c)$ and so
$$\langle \theta(c) \rangle \subseteq \theta({\rm C} (\widehat{K},c))\subseteq
{\rm C} (\widehat{G},\theta(c)) \cong \left\{\begin{array}{ll}
{\rm Z}_{2}^{2}, & \mbox{$q$ even}\\
{\rm Z}_{2}, & \mbox{$q$ odd}
\end{array}
\right.
$$
Now, it is easy to see the following
$$\theta({\rm C}(\widehat{K};\zeta_{j}))=\langle \theta(\zeta_{j})\rangle={\mathbb Z}_{2}, \;
\theta({\rm C}(\widehat{K};\eta_{j}))=\langle \theta(\eta_{j})\rangle={\mathbb Z}_{2},\; \theta({\rm C}(\widehat{K};\sigma_{jk}))=\theta(\langle \sigma_{j}, \sigma_{jk}\rangle)={\mathbb Z}_{2}^{2},$$
$$\theta({\rm C}(\widehat{K};\sigma_{j}))=\theta(\langle \sigma_{j},
\sigma_{j1},\ldots,\sigma_{jf_{j}},\epsilon_{j1},\ldots, \epsilon_{jm_{j}} \rangle)={\mathbb Z}_{2}^{2}.
$$
Finally,  it follows from Theorem \ref{teo1} that
$$m_{1}+m_{2} \leq 2(\alpha + \beta )+\varepsilon +(f_{1}+\cdots +f_{\varepsilon }) \leq
2(\alpha + \beta +\varepsilon )+(f_{1}+\cdots +f_{\varepsilon}) \leq
2\left(\frac{g-1}{q}\right)+4.$$
$\blacksquare$

%%%%%%%%%%%%%%%%%%%
%%%%%%%%%%%%%%%%%%%
\subsection{Proof of part (2)}
Let $q_{ij}$ be to denote the order of $\tau_{i}\tau_{j}$, where $i<j$. By a permutation of the indices, we may
assume that $2 \leq q_{12} \leq q_{13} \leq q_{23}$. As consequence of Part (1)  one has the
inequality
\begin{equation}\label{ec1}
m_{1}+m_{2}+m_{3} \leq \left[ \frac{g-1}{q_{12}}\right] +\left[ \frac{g-1}{q_{13}}\right]+\left[
\frac{g-1}{q_{23}}\right]+6.
\end{equation}

%%%%%%%%%%%%%%
\subsubsection{Case $ g=2$}
In this case, $m_{1}+m_{2}+m_{3} \leq 6$. As $m_{i}+m_{j} \leq 4$, either $m_{1}+m_{2}+m_{3} \leq 5$ or else $m_{1}=m_{2}=m_{3}=2$. Claim \ref{hecho1} ends with the proof for $g=2$.
$\blacksquare$

\begin{claim}\label{hecho1}
The case $m_{1}=m_{2}=m_{3}=2$ is not possible for $g=2$.
\end{claim}
\begin{proof}
Assume there is a Schottky group $\Gamma$ of rank two, with region of discontinuity $\Omega$,  such that the handlebody
$M=({\mathbb H}^{3} \cup \Omega)/\Gamma$ admits three symmetries $\tau_{1}$, $\tau_{2}$ and $\tau_{3}$, each of
them having exactly two connected components of fixed points. Set $H=\langle \tau_{1},\tau_{2},\tau_{3}\rangle$. 
Let $\pi:{\mathbb H}^{3} \cup \Omega \to M$ be the
universal covering induced by $\Gamma$, and $\widehat{K}$ be the
extended Kleinian group obtained by lifting $H$ under $\pi$. There is a surjective homomorphism
$\theta:\widehat{K} \to H$ with $\Gamma$ as its kernel. Let $K_{j}=\theta^{-1}(\langle
\tau_{r},\tau_{s}\rangle)$, where $r \neq s$ and $r,s \in \{1,2,3\}-\{j\}$, and $p_{j}$ be the order of $\tau_{r}\tau_{s}$. 

The extended Kleinian group $K_{j}$ is a dihedral extended Schottky group (of finite index in $\widehat{K}$) with region of discontinuity $\Omega$.
By Theorem \ref{const0}, $K_{j}$ is constructed by using $\alpha$ cyclic groups generated by
reflections, $\beta$ cyclic groups generated by imaginary reflections, $\gamma$ cyclic groups generated by
loxodromic transformations, $\delta$ cyclic groups generated by glide-reflections, $\rho$ cyclic groups generated by an elliptic transformation of finite order,
$\eta$ groups generated by an elliptic transformation of finite order and a loxodromic transformation commuting with it, $\kappa$ groups generated by an elliptic transformation of finite order and a glide-reflection (conjugating the elliptic into its inverse)
 and $\varepsilon$ basic extended dihedral groups
$\Gamma_{1}$,..., $\Gamma_{\varepsilon}$. 

If $\Omega_{i}$ is the  region of discontinuity of $\Gamma_{i}$, then $\Omega_{i}/\Gamma_{i}^{+}$ has genus $g_{i}$, has
$2r_{i}$ cone points with cone orders $|t_{i1}|,|t_{i1}|,\ldots, |t_{ir_{i}}|,|t_{ir_{i}}|$ (each of them bigger tan two) and $2n_{i}$ cone points of order two.

Then, $\Omega/K_{j}^{+}$ has genus
$\widetilde{g}=\alpha+\beta+2(\gamma+\delta+\kappa+\eta)+\varepsilon +g_{1}+\cdots+g_{\varepsilon} -1$,
with $2(r_{1}+\cdots+r_{\varepsilon})$ conical points of orders $|t_{11}|$, $|t_{11}|$, $|t_{12}|$,
$|t_{12}|$\ldots, $|t_{\varepsilon r_{\varepsilon}}|$, $|t_{\varepsilon r_{\varepsilon}}|$, 
$2(n_{1}+\cdots+n_{\varepsilon})$ conical points of order $2$, and $4\rho$ conical points of orders $l_{1},l_{1},l_{1},l_{1},\ldots, l_{\rho},l_{\rho},l_{\rho},l_{\rho}$,
and
$$2=p_{j}(\widetilde{g}-1)+1+2p_{j}\sum_{j=1}^{\rho} \left(1-\frac{1}{l_{j}}\right)+p_{j}\sum_{i=1}^{\varepsilon}\sum_{k=1}^{r_{i}}\left(1-\frac{1}{|t_{ik}|}\right)+
\frac{p_{j}}{2}\sum_{i=1}^{\varepsilon}n_{i} \geq p_{j}(\widetilde{g}-1)+1.$$

It follows that $\widetilde{g} \in \{0,1\}$. 

If $\alpha=\beta=\varepsilon=0$ (so $\delta>0$), then $\widetilde{g}=2(\gamma+\delta+\kappa+\eta)-1$, so $\widetilde{g}=1=\delta+\kappa$ and $\eta=\gamma=0$. It follows that $K_{j}$ is a elementary group (either (i) generated by a glide-reflection or (ii) a glide-reflection and an elliptic element with the same fixed points), a contradiction (as it must contains a Schottky group of rank two).

If $\alpha+\beta+\varepsilon>0$, it must be that $\gamma=\delta=\kappa=\eta=0$. If $\varepsilon=0$, then $2=\alpha+\beta$ and $K_{j}$ will be
elementary group, again a contradiction. So $\varepsilon \geq 1$. As
$\alpha+\beta+\varepsilon +g_{1}+\cdots+g_{\varepsilon} =\widetilde{g}+1\in \{1,2\}$,
we must have $\varepsilon=1$. In fact, if $\varepsilon \geq 2$, then $\varepsilon=2$, $\alpha=\beta=g_{1}=g_{2}=0$, $\widetilde{g}=1$ and
$$1=p_{j}\sum_{i=1}^{2}\sum_{k=1}^{r_{i}}\left(1-\frac{1}{|t_{ik}|}\right)+
\frac{p_{j}}{2}\sum_{i=1}^{2}n_{i}.$$
If, for instance, $n_{1} \geq 1$, then (as $p_{j} \geq 2$) $r_{1}=r_{2}=0=n_{2}$ and $n_{1}=1$, from which we obtain that $K_{j}$ is elementary, a contradiction. In particular, $n_{1}=n_{2}=0$. The other case can be worked similarly to obtain a contradiction.

Assume now that $\varepsilon=1$. Then $\alpha+\beta+g_{1} \in \{0,1\}$. If $g_{1}=1$, then $\alpha=\beta=0$ and $K_{j}$ is again elementary. So, $g_{1}=0$
and $\widetilde{g}=\alpha+\beta \in \{0,1\}$, that is
$(\alpha,\beta) \in \{(0,0), (1,0), (0,1)\}.$

In the case $(\alpha,\beta)=(0,0)$, we have $\widetilde{g}=0$ and$$2 \geq 1-n
+n\sum_{k=1}^{r_{1}}\left(1-\frac{1}{|t_{ik}|}\right)+ \frac{nn_{1}}{2}.$$

As each $|t_{ij}| \geq 2$, the above ensures that
$2 \geq n(n_{1}+r_{1}-2).$
Now, if $n_{1}+r_{1} \geq 3$, then $n=2$ (as required) and, necessarily, $n_{1}+r_{1}=3$. If $n_{1}+r_{1} \leq
2$, then $K_{j}$ will be elementary, a contradiction.
In the case $(\alpha,\beta) \in \{(1,0),(0,1)\}$, we have $\widetilde{g}=1$
and $$2 \geq 1+n\sum_{k=1}^{r_{1}}\left(1-\frac{1}{|t_{ik}|}\right)+ \frac{nn_{1}}{2}.$$

Again, as $|t_{1k}| \geq 2$, one obtains that 
$2 \geq n(n_{1}+r_{1}).$ Then $n=2$ and $n_{1}+r_{1}=1$, but again this makes $K_{j}$ to be elementary, a contradiction.

\smallskip
As a consequence of the above, we see that $K_{j}$ is generated by a reflection $\rho_{j1}$ and three other
involutions $\rho_{j2}$, $\rho_{j3}$, $\rho_{j4}$, each one commuting with $\rho_{j1}$ (some of the involutions
may be reflections and others may be imaginary reflections). Note also that $K_{j}$ keeps invariant the circle of
fixed points of $\rho_{j1}$, that is, the limit set of $K_{j}={\mathbb Z}_{2} \oplus ({\mathbb Z}_{2}*{\mathbb
Z}_{2}*{\mathbb Z}_{2})$ is contained in the circle of fixed points of $\rho_{j1}$. As $K_{j}$ has finite index in $K$, the limit set of $K$ is contained in such a circle, and 
as the limit set of
$\widehat{K}$ is infinite, such a circle is uniquely determined by it. As a consequence,
$\rho_{11}=\rho_{21}=\rho_{31}=\rho$. This asserts that $\theta(\rho) \in \theta(K_{1}) \cap \theta(K_{2}) \cap
\theta(K_{3})$. 

As, by the definition of the groups $K_{j}$'s, (i) $\tau_{j} \notin \theta(K_{1}) \cap \theta(K_{2})
\cap \theta(K_{3})$, (ii) $\theta(\rho)$ is a symmetry and (iii) the symmetries in $H$ are exactly $\tau_{1}$, $\tau_{2}$,
$\tau_{3}$ and $\tau_{1}\tau_{2}\tau_{3}$, the only possibility is to have
$\theta(\rho)=\tau_{1}\tau_{2}\tau_{3}$, from which, for instance, $\tau_{1}\tau_{2}\tau_{3} \in \langle
\tau_{1},\tau_{2}\rangle$; obligating to have that $\tau_{3} \in \langle \tau_{1},\tau_{2}\rangle$, a
contradiction.
\end{proof}

%%%%%%%%%%%%%%%
\subsubsection{Case $g=3$}
If $q_{23} \geq 3$, then $m_{2}+m_{3} \leq 4$. Since $m_{1} \leq g+1=4$, we get in this case that
$m_{1}+m_{2}+m_{3} \leq 8$.
Let us now assume $q_{12}=q_{13}=q_{23}=2$, in which case $\langle \tau_{1},\tau_{2},\tau_{3}\rangle \cong
{\mathbb Z}_{2}^{3}$. We may reorder again the indices to assume that $m_{1} \leq m_{2} \leq m_{3} \leq 4$. If
$m_{3}=4$, then inequality \eqref{ec1} asserts that $m_{1}+m_{2} \leq 2$, so $m_{1}+m_{2}+m_{3} \leq 8$.  Now,
the only case with $m_{3} \leq 3$ for which we do not have $m_{1}+m_{2}+m_{3} \leq 8$ is when
$m_{1}=m_{2}=m_{3}=3$. Claim \ref{claim5.2} ends the proof for $g=3$.
$\blacksquare$

\begin{claim}\label{claim5.2}
The case $m_{1}=m_{2}=m_{3}=3$ is not possible for $g=3$.
\end{claim}
\begin{proof}The Proof follows the same ideas as for the Claim \ref{hecho1}. Let $\Gamma$ be a Schottky group of rank $g=3$, with region of discontinuity $\Omega$,  so that the handlebody
$M=({\mathbb H}^{3} \cup \Omega)/\Gamma$ admits three symmetries $\tau_{1}$, $\tau_{2}$ and $\tau_{3}$, each of
them having exactly three connected components of fixed points, and with $H=\langle
\tau_{1},\tau_{2},\tau_{3}\rangle={\mathbb Z}_{2}^{3}$. Let $\pi:{\mathbb H}^{3} \cup \Omega \to M$ be the universal covering induced by $\Gamma$ and let $\widehat{K}$
be the extended Kleinian group obtained by lifting $H$ under $\pi$. There is a surjective homomorphism
$\theta:\widehat{K} \to H$ with $\Gamma$ as its kernel. Let $K_{j}=\theta^{-1}(\langle
\tau_{r},\tau_{s}\rangle)$, where $r \neq s$ and $r,s \in \{1,2,3\}-\{j\}$; so $\theta(K_{j})={\mathbb
Z}_{2}^{2}$. The extended Kleinian group $K_{j}$ is a dihedral extended Schottky group and has index two in $\widehat{K}$,
so it has region of discontinuity $\Omega$. As consequence of Theorem \ref{const0}, $K_{j}$ is constructed by using $\alpha$ cyclic groups generated by
reflections, $\beta$ cyclic groups generated by imaginary reflections, $\gamma$ cyclic groups generated by
loxodromic transformations, $\delta$ cyclic groups generated by glide-reflections, $\rho$ cyclic groups generated by an elliptic transformation of order two,
$\eta$ groups generated by an elliptic transformation of order two, $\kappa$ groups generated by an elliptic transformation of order two and a glide-reflection
 and $\varepsilon$ basic extended dihedral groups
$\Gamma_{1}$,..., $\Gamma_{\varepsilon}$, so that $\alpha+\beta+\delta
+\varepsilon>0$. 
 
As before, $\Omega/K_{j}^{+}$ will have genus
$\widetilde{g}=\alpha+\beta+2(\gamma+\delta+\kappa+\eta)+\varepsilon+g_{1}+\cdots+g_{\varepsilon} -1,$
with $2(2\rho+r_{1}+\cdots+r_{\varepsilon}+n_{1}+\cdots+n_{\varepsilon})$ conical points of order $2$,  and so that
$$2 \geq 2(\widetilde{g}-1)+\sum_{i=1}^{\varepsilon}(r_{i}+n_{i}).$$

In particular, $\widetilde{g} \in \{0,1\}$.
If $\alpha=\beta=\varepsilon=0$ (so $\delta+\kappa>0$), then $\widetilde{g}=2(\gamma+\delta+\kappa+\eta)-1$, so $\gamma=\eta=0$ and $\delta+\kappa=1$, from which we obtain that $K_{j}$ is elementary, a contradiction.
As $\alpha+\beta+\varepsilon>0$, then $\gamma=\delta=\kappa=\eta=0$. If $\varepsilon=0$, the same
asserts that $K_{j}$ will be elementary group containing the non-elementary group $\Gamma$, a contradiction, so
$\varepsilon \geq 1$. If $\varepsilon \geq 3$, then $\alpha+\beta<0$, a contradiction.
If $\varepsilon=2$, then $\alpha+\beta=\widetilde{g}-1$, so the only possible case is to
have $\widetilde{g}=1$ and $\alpha=\beta=0$. In this way, 
$$(\alpha,\beta,\varepsilon;\widetilde{g}) \in
\{(0,0,1;0), (1,0,1;1), (0,1,1;1), (0,0,2;1)\}.$$

Let us consider the case $(\alpha,\beta,\varepsilon;\widetilde{g})=(0,0,2;1)$. In this case
$K_{j}$ is free product of two groups $\Gamma_{1}=\langle \rho_{1},\eta_{1}\rangle={\mathbb Z}_{2}^{2}$ and
$\Gamma_{2}=\langle \rho_{2},\eta_{2}\rangle={\mathbb Z}_{2}^{2}$, where $\rho_{k}$ is reflection and $\eta_{k}$
is either a reflection or an imaginary reflection. In this case a complete set of symmetries of $K_{j}$ is given
by $\{\rho_{1},\rho_{2},\eta_{1},\eta_{2}\}$. Let us note that $\theta(\rho_{k}) \neq \theta(\eta_{k})$ since
otherwise the elliptic element of order two $\rho_{k}\eta_{k} \in \Gamma$, a contradiction to the fact that
$\Gamma$ is torsion free. As
$${\rm C}(K_{j},\rho_{k})={\rm C}(K_{j},\eta_{k})=\langle
\rho_{k},\eta_{k}\rangle={\mathbb Z}_{2}^{2}$$ and that $\theta(\rho_{k}) \neq \theta(\eta_{k})$, we may see from
Theorem \ref{teo1} that the number of fixed points of any of the symmetries $\theta(\rho_{1})$ and
$\theta(\eta_{1})$ is at most $2$, a contradiction to the fact we are assuming the symmetries $\tau_{1}$,
$\tau_{2}$ and $\tau_{3}$ each one has exactly $3$ components.

Let us consider the case $(\alpha,\beta,\varepsilon; \widetilde{g})=(0,0,1;0)$. In this case
$K_{j}=\langle \rho, \eta_{1},\eta_{2},\eta_{3},\eta_{4}\rangle$, where $\rho$ is a reflection and each of the
$\eta_{k}$ is either a reflection or an imaginary reflection commuting with $\rho$ (each element of $K_{j}$ keeps
invariant the circle of fixed points of $\rho$). As before, $\theta(\rho) \neq \theta(\eta_{k})$, for each
$k=1,2,3,4$. So, in particular, $\theta(\eta_{1})=\theta(\eta_{2})=\theta(\eta_{3})=\theta(\eta_{4})$. In this
case, a complete set of symmetries of $K_{j}$ is given by $\{\rho, \eta_{1},\eta_{2},\eta_{3},\eta_{4}\}$ and
$${\rm C}(K_{j},\rho)=K_{j}, \; {\rm C}(K_{j},\eta_{k})=\langle \rho, \eta_{k}\rangle={\mathbb Z}_{2}^{2}$$
and it follows, from Theorem \ref{teo1}, that $\theta(\rho)$ has at most $1$ connected components
of fixed points, a contradiction to the assumption the symmetries have $3$ conected components.

Let now $(\alpha,\beta,\varepsilon; \widetilde{g})=(1,0,1;0)$ (similar
arguments for the case $(\alpha,\beta,\varepsilon;\widetilde{g})=(0,1,1;0)$). In this case
$K_{j}$ is a free product of a cyclic group generated by a reflection $\zeta$ and a group $K_{j}^{0}= \langle
\rho, \eta_{1},\eta_{2},\eta_{3},\eta_{4}\rangle$, where $\rho$ is a reflection and each of the $\eta_{k}$ is
either a reflection or an imaginary reflection commuting with $\rho$ (each element of $K_{j}^{0}$ keeps invariant
the circle of fixed points of $\rho$). Again, as before, $\theta(\rho) \neq \theta(\eta_{k})$, for each
$k=1,2,3,4$. So, in particular, $\theta(\eta_{1})=\theta(\eta_{2})=\theta(\eta_{3})=\theta(\eta_{4})$. In this
case, a complete set of symmetries of $K_{j}$ is given by $\{\zeta, \rho, \eta_{1},\eta_{2},\eta_{3},\eta_{4}\}$
and
$${\rm C}(K_{j},\zeta)=\langle \zeta\rangle={\mathbb Z}_{2}, \; {\rm C}(K_{j},
\rho)=K_{j}^{0}, \; {\rm C}(K_{j},\eta_{k})=\langle \rho, \eta_{k}\rangle={\mathbb Z}_{2}^{2}.$$

The only way for both symmetries in $\theta(K_{j})$ to have exactly $3$ connected components of fixed points is
to have that $\theta(\zeta)=\theta(\rho)$. But in this case we should have that $\Gamma$ is the Schottky group
generated by the loxodromic transformations
$\rho\zeta, \eta_{4}\eta_{1}, \eta_{4}\eta_{2}, \eta_{4}\eta_{3},$
which is of rank $4$, a contradiction.

\end{proof}

\subsubsection{Case $g\geq 4$}
In this case, inequality \eqref{ec1} asserts
\begin{equation}\label{ec2}
m_{1}+m_{2}+m_{3} \leq \left( \frac{1}{q_{12}}+\frac{1}{q_{13}}+\frac{1}{q_{23}}\right)(g-1)+6.
\end{equation}

If $q_{12}^{-1}+q_{13}^{-1}+q_{23}^{-1}\leq 1$, then the above ensures that
$m_{1}+m_{2}+m_{3} \leq g+5$.
If $q_{12}^{-1}+q_{13}^{-1}+q_{23}^{-1}> 1$, then we have the following cases:
$$
\begin{array}{|c|c|c|c|}\hline
&&&\\[-3mm]
q_{12} & q_{13} & q_{23} & H=\langle \tau_{1},\tau_{2},\tau_{3}\rangle \\
&&&\\[-3mm]
\hline
2& 2 & r \geq 2 & {\mathbb Z}_{2} \times {\mathbb D}_{r} \\\hline
2 & 3 & 3 & {\mathbb Z}_{2} \ltimes {\mathcal A}_{4} \\\hline
2 & 3 & 4 & {\mathbb Z}_{2} \ltimes {\mathfrak S}_{4} \\\hline
2 & 3 & 5 & {\mathbb Z}_{2} \ltimes {\mathcal A}_{5} \\\hline
\end{array}
$$
where ${\mathbb Z}_{2}=\langle \tau_{1}\rangle$, ${r}=\langle \tau_{2},\tau_{3}\rangle$ and ${\mathcal A}_{4}$,
${\mathfrak S}_{4}$ and ${\mathcal A}_{5}$ are generated by $\langle \tau_{1}\tau_{2}, \tau_{1}\tau_{3}\rangle$
in each case. In the cases $q_{13}=3$ one has that $\langle \tau_{1},\tau_{3}\rangle \cong {\mathbb D}_3    $, so $\tau_{1}$ and
$\tau_{3}$ are conjugated, that is, $m_{1}=m_{3}$. It follows that
$$2m_{1}=2m_{3}=m_{1}+m_{3} \leq 2 \left[ \frac{g-1}{3}\right]+4,$$
so $\displaystyle{m_{1}=m_{3} \leq \left[ \frac{g-1}{3}\right]+2}$. As $q_{23} \geq 3$, 
$$m_{1}+(m_{2}+m_{3}) \leq \left(\left[ \frac{g-1}{3}\right]+2\right) +
\left( 2 \left[ \frac{g-1}{q_{23}}\right]+4\right)\leq 3\left[ \frac{g-1}{3}\right]+6 \leq g+5.$$

Now, for the case $q_{12}=q_{13}=2$ and $q_{23}=r \geq 2$, inequality \eqref{ec1} asserts
$$m_{1}+m_{2}+m_{3} \leq 2\left[ \frac{g-1}{2}\right]+\left[ \frac{g-1}{r}\right]+6 \leq 2
\left( \frac{g-1}{2}\right)+\frac{g-1}{r}+6=\frac{(r+1)g+5r-1}{r}.$$

%%%%%%%%%%%%%%%%
%%%%%%%%%%%%%%%%
\section{Examples}\label{Sec:4}

\begin{example}[Sharp upper bound in part (1) of Theorem \ref{sumak=2}]
Let $q \geq 2$ and let us consider an extended Schottky group $K$, constructed by Theorem
\ref{maintheo} using exactly $r+1$ reflections $E_{1}$,\ldots , $E_{r+1}$. The orbifold uniformized
by  $K$ is a planar surface bounded by exactly $r+1$ boundary loops. Let us consider the surjective
homomorphism
$$\theta:K \to {\rm D}_{q}=\langle x,y: x^{2}=y^{2}=(yx)^{q}=1\rangle: \;
\theta(E_{1})=x, \; \theta(E_{2})=\cdots=\theta(E_{r+1})=y.$$

Let $\Gamma=\ker(\theta)$. If we set $L=E_{2}E_{1}$ and $C_{j}=E_{r+1}E_{j+1}$, for $j=1,\ldots ,r-1$, then it is
not difficult to see that
$$\Gamma=\langle L^{q},C_{1},\ldots ,C_{r-1}, LC_{1}L^{-1},\ldots ,LC_{r-1}L^{-1},\ldots , L^{q-1}C_{1}L^{-q+1},
\ldots ,L^{q-1}C_{r-1}L^{-q+1} \rangle$$
is a Schottky group of rank $g=(r-1)q+1$. Let us consider the extended Schottky groups
$$\Gamma_{1}=\theta^{-1}(\langle x\rangle)=\langle E_{1}, \Gamma\rangle, \quad 
\Gamma_{2}=\theta^{-1}(\langle y\rangle)=\langle E_{2}, \Gamma\rangle.$$

As the group $K$ contains no imaginary reflections nor real Schottky groups, it follows that
$\Gamma_{j}$ is constructed, by Theorem \ref{maintheo}, using $\alpha_{i}$ reflections and $\beta_{i}$ loxodromic
and glide-reflection transformations. The handlebody $M_{\Gamma}$ admits two symmetries, say $\tau_{1}$ and $\tau_{2}$
induced by $\Gamma_{1}$ and $\Gamma_{1}$ respectively. The number of connected components of fixed points of $\tau_{i}$ is
exactly $\alpha_{i}$. As consequence of Theorem \ref{sumak=2}, we should have $\alpha_{1}+\alpha_{2} \leq
2(r+1)$. Next, we proceed to see that in fact we have an equality, showing that the upper bound in Theorem
\ref{sumak=2} is sharp. In order to achieve the above, we use Theorem \ref{teo1}. A complete set of symmetries in
$K$ is provided by $E_{1}$,\ldots , $E_{r+1}$. We also should note that ${\rm C}
(K,E_i)=\langle E_{i}\rangle$, for every $j$, and that $J(i)=\emptyset$ and $I(i)=F(i)$.

\subsubsection*{Case $q$ odd}
In this case, $I(1)=\{1,2,\ldots ,r+1\}$ and ${\rm C} ({\mathbb D}_{q},x)=\langle x
\rangle$. It follows, from Theorem \ref{teo1}, that
$\alpha_{1}=\alpha_{2}=r+1$
and we are done.

\subsubsection*{Case $q$ even}
In this case, $I(1)=\{1\}$, $I(2)=\{2,\ldots ,r+1\}$, ${\rm C} ({\mathbb
D}_{q},x)=\langle x, (yx)^{q/2} \rangle\cong {\mathbb Z}_{2}^{2}$ and ${\rm C}
({\mathbb D}_{q},y)=\langle y, (yx)^{q/2} \rangle\cong {\mathbb Z}_{2}^{2}$. It
follows, from Theorem \ref{teo1}, that
$\alpha_{1}=2, \; \alpha_{2}=2r$
and we are done.
\end{example}

%%%%%%%%%%%%%%%%%
\begin{example}[Sharp upper bounds in part (2) of Theorem \ref{sumak=2} for $g=2$]
Let $K$ be the extended Kleinian group generated by four reflections, say $\eta_{1}$,\ldots,
$\eta_{4}$, where $\eta_{1}(z)=\overline{z}$, $\eta_{2}(z)=-\overline{z}$, $\eta_{4}(z)=1/\overline{z}$ and
$\eta_{3}$ is the reflection on a circle $\Sigma$ which is orthogonal to the unit circle and disjoint from the
real and imaginary axis. One may see that
$$K=\langle \eta_{4} \rangle \times ( \langle \eta_{1},\eta_{2} \rangle \ast
\langle \eta_{3} \rangle) \cong {\mathbb Z}_{2} \times ({\mathbb D}_{2} * {\mathbb Z}_{2}).$$

If $\Gamma=\langle A=\eta_{1}\eta_{3}, B=\eta_{1}\eta_{2}\eta_{3}\eta_{2}\rangle$, then it is a Schottky group of rank $2$ 
with a fundamental domain bounded by the $4$ circles
$C_{1}=\Sigma$, $C'_{1}=\eta_{1}(\Sigma)$, $C_{2}=\eta_{2}(\Sigma)$ and $C'_{2}=\eta_{2}\eta_{1}(\Sigma)$ such that
$A(C_{1})=C'_{1}$ and $B(C_{2})=C'_{2}$.
As $\eta_{1} A \eta_{1}=\eta_{3} A \eta_{3}=A^{-1}$, $\eta_{2}A\eta_{2}=\eta_{4}B\eta_{4}=B$,
$\eta_{4}A\eta_{4}=\eta_{2}B\eta_{2}=A$, $\eta_{1}B\eta_{1}=B^{-1}$ and $\eta_{3}B\eta_{3}=A^{-1}B^{-1}A$, it
follows that $\Gamma$ is a normal subgroup of $K$. Moreover, $K/\Gamma \cong {\mathbb
Z}_{2}^{3}={\mathbb Z}_{2} \times {\rm D}_{2}.$
On the handlebody $M_{\Gamma}$ both $\eta_{1}$ and $\eta_{3}$ induce the same symmetry
$\tau_{1}$ with exactly $3$ connected components of fixed points (each of them a disc), $\eta_{2}$ induces a
symmetry $\tau_{2}$ with exactly one connected component of fixed points (a dividing disc) and $\eta_{4}$ induces
a symmetry $\tau_{3}$ with exactly one connected component of fixed points (this being an sphere with three
borders). It follows that the three induced symmetries are non-conjugated.
\end{example}

%%%%%%%%%%%%%%%%%%%%%
\begin{example}[Sharp upper bounds in part (2) of Theorem \ref{sumak=2} for $g = 3$]
Let $K$ be the extended Kleinian group generated by three reflections,
$\eta_{1}(z)=\overline{z}$, $\eta_{2}(z)=-\overline{z}$ and $\eta_{3}$ the reflection on a circle $\Sigma$
disjoint from the real and imaginary lines. One has that
$$K=\langle \eta_{1}, \eta_{2} \rangle * \langle \eta_{3} \rangle \cong {\rm D}_2 * {\mathbb Z}_{2}.$$

The group $\Gamma=\langle A_{1}=(\eta_{3}\eta_{1})^{2},A_{2}=(\eta_{3}\eta_{2})^{2},
A_{3}=\eta_{3}\eta_{2}\eta_{1}\eta_{3}\eta_{1}\eta_{2}\rangle$ is a
Schottky group of rank $3$ with a fundamental domain bounded by the $6$ circles $C_{1}=\eta_{1}(\Sigma)$,
$C'_{1}=\eta_{3}(C_{1})$, $C_{2}=\eta_{2}(\Sigma)$, $C'_{2}=\eta_{3}(C_{2})$, $C_{3}=\eta_{2}(C_{1})$ and
$C'_{3}=\eta_{3}(C_{3})$,  such that $A_{1}(C_{1})=C'_{1}$, $A_{2}(C_{2})=C'_{2}$ and $A_{3}(C_{3})=C'_{3}$.
Similarly to the previous case, one may check that $\Gamma$ is a normal subgroup of $K$ and that 
$K/\Gamma \cong {\mathbb Z}_{2}^{3}$. The reflection $\eta_{j}$ induces a symmetry $\tau_{j}$ (for each
$j=1,2,3$) on the handlebody $M_{\Gamma}$. In this case (either by direct inspection or by  using
Theorem \ref{teo1}) each of $\tau_{1}$ and $\tau_{2}$ has exactly $2$ connected components of fixed
points, and $\tau_{3}$ has $4$ connect components of fixed points. In some cases the handlebody $M_{\Gamma}$ will have
extra automorphisms conjugating $\tau_{1}$ with $\tau_{2}$ (for instance, when $\Sigma$ is orthogonal to the line
$L=\{Re(z)=Im(z) \}$); but in the generic case this will not happen (that is, the three of them will be
non-conjugated).
\end{example}

%%%%%%%%%%%%%%%%%%
\begin{example}[Sharp upper bounds in part (2) of Theorem \ref{sumak=2} for $g \geq 4 $ and $H \ncong {\mathbb Z}_{2} \times {\mathbb D}_r $]
Let $K$ be an extended Schottky group constructed by using $2n+3$ reflections (or imaginary reflections
or combination of them), say $\eta_{1}$,\ldots, $\eta_{2n+3}$. Consider the surjective homomorphism
$\theta:K \to {\mathbb D}_{3} =\langle a,b: a^{2}=b^{2}=(ab)^{3}=1\rangle$, defined by  
$\theta(\eta_{1})=\cdots =\theta(\eta_{2n+2})=a$, $\theta(\eta_{2n+3})=b$.
In this case $\Gamma=\ker(\theta)$ is a Schottky group of rank $g=6n+4$.  The handlebody $M_{\Gamma}$
admits the symmetries $\tau_{1}=a$, $\tau_{2}=b$ and $\tau_{3}=bab$. By direct inspection (or
by using Theorem \ref{teo1}) it can be checked that $m_{1}=m_{2}=m_{3}=2n+3$.
\end{example}

%%%%%%%%%%%%%%%%%%%%%
\begin{example}[Sharp upper bounds in part (2) of Theorem \ref{sumak=2} for $g \geq 4$ and $H \cong {\mathbb Z}_{2} \times {\mathbb D}_r$]
Let $K$ be an extended Schottky group constructed by using $3$ reflections (or imaginary reflections or
combination of them), say $\eta_{1}$, $\eta_{2}$ and $\eta_{3}$. Consider the surjective
homomorphism
$\theta:K \to {\mathbb Z}_{2} \times {\mathbb D}_3  =\langle c\rangle \times \langle a,b:
a^{2}=b^{2}=(ab)^{3}=1\rangle$, defined by  
$\theta(\eta_{1})=c$, $\theta(\eta_{2})=a$ and $\theta(\eta_{3})=b$.
In this case $\Gamma=\ker(\theta)$ is a Schottky group of rank $g=2r+1$.  The handlebody $M_{\Gamma}$
admits the symmetries $\tau_{1}=c$, $\tau_{2}=a$ and $\tau_{3}=b$. In this case, $m_{1}=2r$ and $m_{2}=m_{3}=4$.
\end{example}

%%%%%%%%%%%%%%%%%%%%
%%%%%%%%%%%%%%%%%%%%


\begin{thebibliography}{99}


\bibitem{Bers}
L. Bers.
Automorphic forms for Schottky groups,
{\it Adv. in Math.} {\bf 16} (1975), 332--361.

%\bibitem{BCGG}
%E. Bujalance , F. J. Cirre , J. M. Gamboa  and G. Gromadzk.
%On the number of ovals of a symmetry of a compact Riemann surface.
%{\it Rev. Mat. Iberoamericana} {\bf 24} No. 2 (2008), 391--405.

\bibitem{BCS}
E. Bujalance, A. F. Costa and D. Singerman.
Applications of Hoare's theorem to symmetries of Riemann surfaces. 
{\it Ann. Acad. Sci. Fenn. Ser. A I Math.} {\bf 18} (1993), 307--322.


%\bibitem{GG1}
%G. Gromadzki.
%On ovals on Riemann surfaces.
%{\it Revista Matem\'atica Iberoamericana} {\bf 16} (2000), 515--527.

%\bibitem{GG3}
%G. Gromadzki. 
%On a Harnack--Natanzon theorem for the family of real forms of Riemann surfaces. 
%{\it J. Pure Appl. Algebra} {\bf 121} No. 3 (1997), 253--269.

%\bibitem{GG2}
%G. Gromadzki and E. Koz\l owska-Walania.
%On ovals of non-conjugate symmetries of Riemann surfaces.
%{\it Internat. J. Math.} {\bf  20} No. 1 (2009), 1--13.



\bibitem{H-G:ExtendedSchottky}
G. Gromadzki and R. A. Hidalgo.
Schottky uniformizations of Symmetries.
{\it Glasgow Mathematical Journal} {\bf 55} (2013), 591--613.



%\bibitem{GI1}
%G. Gromadzki and M. Izquierdo. 
%On ovals of Riemann surfaces of even genera. 
%{\it Geometriae Dedicata} {\bf 78} (1999), 81--88.


\bibitem{Harnack}
A. Harnack.
\"Uber die Vieltheiligkeit der ebenen algebraischen Kurven.
{\it Math. Ann.} {\bf 10} (1876), 189--199.


\bibitem{Hidalgo:MEFP}
R. A. Hidalgo.
The mixed elliptically fixed point property for Kleinian groups.
{\it Ann. Acad. Scie. Fenn.} {\bf 19} (1994), 247--258.


\bibitem{Hidalgo:Auto}
R. A. Hidalgo.
Automorphisms groups of Schottky type.
{\it Ann. Acad. Sci. Fenn.} {\bf 30} (2005), 183--204.


\bibitem{H-M:imaginary}
R. A. Hidalgo, B. Maskit.
Fixed points of imaginary reflections on hyperbolic handlebodies.
{\it Math. Proc. Camb. Phil. Soc.} {\bf 148} (2010), 135--158.

\bibitem{HS}
R. A. Hidalgo and S. Sarmiento.
Real structures on marked Schottky space.
{\it Journal of the London Math. Soc.} {\bf 98} (2) (2018), 253--274. 


%\bibitem{IS}
%M. Izquierdo and D. Singerman.
%Pairs of symmetries of Riemann surfaces.
%{\it Ann. Acad. Sci. Fenn. Math.} {\bf 23} (1998), 3--24.



\bibitem{Ka-Mc}
J. Kalliongis and D.  McCullough.
Orientation-reversing involutions on handlebodies.
{\it Trans. of the Math. Soc.} (5) {\bf 348} (1996), 1739--1755.


\bibitem{Koebe}
P. Koebe.
\"{U}ber die Uniformisierung der Algebraischen Kurven II,
{\it Math. Ann.} {\bf 69} (1910), 1--81.


\bibitem{Ewa}
E. Kozlowska-Walania. 
On commutativity and ovals for a pair of symmetries of a Riemann surface. 
{\it Colloq. Math.} {\bf 109} (2007), 61--69.



\bibitem{MacBeath}
A. M. Macbeath. 
The classification of non-Euclidean crystallographic groups. 
{\it Canad. J. Math.} {\bf 19} (1967) 1192--1205.



\bibitem{Maskit:Comb}
B. Maskit.
On Klein's combination theorem III.
Advances in the Theory of Riemann Surfaces (Proc. Conf., Stony Brook, N.Y., 1969),
{\it  Ann. of Math. Studies} {\bf 66} (1971), 
Princeton Univ. Press, 297-316.


\bibitem{Maskit:Comb4}
B. Maskit.
On Klein's combination theorem. IV. 
{\it Trans. Amer. Math. Soc.} {\bf 336} (1993), 265--294.


\bibitem{Maskit:book}
B. Maskit.
{\it Kleinian Groups},
GMW, Springer-Verlag, 1987.


\bibitem{Maskit:Schottky groups}
B. Maskit.
A characterization of Schottky groups,
{\it J. d'Analyse Math.} {\bf 9} (1967), 227--230.

\bibitem{Maskit:function2}
B. Maskit.
Decomposition of certain Kleinian groups.
{\it Acta Math.} {\bf 130} (1973), 243--263.



\bibitem{Maskit:function3}
B. Maskit. 
On the classification of Kleinian Groups I. Koebe groups. 
{\it Acta Math.} {\bf 135} (1975), 249--270.



\bibitem{Maskit:function4}
B. Maskit. 
On the classification of Kleinian Groups II. Signatures. 
{\it Acta Math.} {\bf 138} (1976), 17--42.


\bibitem{MT}
K. Matsuzaki and M. Taniguchi. 
{\it Hyperbolic Manifolds and Kleinian Groups}. Oxford Mathematical Monographs.
Oxford Science Publications. The Clarendon Press, Oxford University Press, New York, 1998.



\bibitem{M-Y}
W. H. Meeks III and S.-T. Yau.
Topology of three-dimensional manifolds and the embedding problem in minimal surface theory.
{\it Ann. of Math.} (2) {\bf 112} (1980), 441--484.

\bibitem{Nag}
S. Nag.
{\it The complex analytic theory of Teichm\"uller spaces.}
A Wiley-Interscience Publication. John Wiley \& Sons, Inc., New York 1988.


\bibitem{Natanzon}
S. M. Natanzon.
Automorphisms of the Riemann surface of an $M$-curve, (Russian)
{\it Funktsional. Anal. i Prilozhen}. {\bf 12}, No. 3 (1978), 82--83.


%\bibitem{Natanzon2}
%S. M. Natanzon. 
%On the total number of ovals of real forms of complex algebraic curves. 
%{\it Uspekhi Mat. Nauk} (1) {\bf 35}, 1980, 207--208. ({\it Russian Math. Surveys} (1) {\bf 35}, 1980, 22--224.)


\bibitem{N1}
S. M. Natanzon.
Finite groups of homeomorphisms of surfaces and real forms of complex algebraic curves.
{\it Trans. Moscow Math. Soc.} (1989), 1--51.

%\bibitem{Singerman}
%D. Singerman.
%Mirrors on Riemann surfaces. 
%Second International Conference on Algebra (Barnaul, 1991), 411--417, {\it Contemp. Math.} {\bf184}, Amer. Math. Soc., Providence, RI, 1995.

\bibitem{Weil}
A. Weil.
The field of definition of a variety.
{\it Amer. J. Math.} {\bf 78} (1956), 509-524.



\bibitem{Z1}
B. Zimmermann.
On maximally symmetric hyperbolic 3-manifolds. 
Progress in Knot Theory and Related Topics, 143--153, Travaux En Cours, 56, Hermann, Paris 1997.


\bibitem{Z2}
B. Zimmermann. 
\"Uber Hom\"oomorphismen $n$-dimensionaler Henkelk\"orper und
endliche Erweiterungen von Schottky-Gruppen (German), [On homeomorphisms of
$n$-dimensional handlebodies and on finite extensions of Schottky groups] {\it
Comment. Math. Helv.} {\bf 56} (1981), no. 3, 474--486.




\end{thebibliography}
\end{document}